\documentclass[a4paper,12pt]{amsart}

\usepackage[english]{babel}
\usepackage{amsmath,amsthm,amssymb,amsfonts}
\usepackage[all]{xy}
\usepackage{multicol}
\usepackage{todonotes}
 \usepackage{enumitem} 
\usepackage[dvipsnames,table]{xcolor}
\usepackage{pdflscape}
\usepackage[bookmarks=true,hyperindex,pdftex,colorlinks,citecolor=cyan]{hyperref}
\usepackage{mathtools}
\usepackage{extarrows}
\usepackage{soul}

\hypersetup{colorlinks=true,  linkcolor=blue, citecolor=red, urlcolor=cyan}
\usepackage{cleveref}
\usepackage[export]{adjustbox}
\usepackage{graphicx}
\usepackage{subcaption}
\usepackage{tikz}
\usepackage{tikz-cd}
\usetikzlibrary{arrows.meta,calc,matrix}
\usetikzlibrary{arrows}
\usetikzlibrary{arrows.meta, positioning}
\usepackage{caption}
\usepackage[normalem]{ulem}
\usepackage{marginnote}
\usepackage{cancel}

\usepackage{lscape}

\setlist[enumerate]{itemsep=0.2em, topsep=0.25em}

\makeatletter
\newcommand*\bigcdot{\mathpalette\bigcdot@{.5}}
\newcommand*\bigcdot@[2]{\mathbin{\vcenter{\hbox{\scalebox{#2}{$\m@th#1\bullet$}}}}}
\makeatother

 \newcommand{\vertiii}[1]{{\left\vert\kern-0.25ex\left\vert\kern-0.25ex\left\vert #1 
		\right\vert\kern-0.25ex\right\vert\kern-0.25ex\right\vert}}

\usepackage{tikz}
\usetikzlibrary{mindmap,backgrounds, calc}
\usetikzlibrary{matrix,chains,scopes,positioning,arrows,fit}
\usetikzlibrary{positioning, shapes, patterns}
\usetikzlibrary{automata} 
\usetikzlibrary{shapes.geometric, arrows, arrows.meta}
\usetikzlibrary{positioning,decorations.markings}

\usepackage[a4paper,left=3cm,right=3cm,top=2.5cm,bottom=3cm]{geometry}

\theoremstyle{plain}
\newcounter{maincoro}

\newtheorem{theorem}{Theorem}[section]
\newtheorem{lemma}[theorem]{Lemma}

\newtheorem{proposition}[theorem]{Proposition}
\newtheorem{corollary}[theorem]{Corollary}

\newtheorem*{corollary*}{Corollary}

\newtheorem{thmx}{Theorem}

\theoremstyle{definition}
\newcounter{maintheorem}

\newtheorem{definition}[theorem]{Definition}
\newtheorem{example}[theorem]{Example}

\theoremstyle{remark}
\newtheorem{remark}[theorem]{Remark}
\numberwithin{equation}{section}

\newcommand{\R}{\mathbb{R}}
\newcommand{\C}{\mathbb{C}}
\newcommand{\N}{\mathbb{N}}

\newcommand{\Z}{\mathbb{Z}}

\newcommand{\T}{\mathbb{T}}

\newcommand{\FF}{\mathcal{F}}

\newcommand{\Ss}{\mathcal{S}}

\newcommand{\Span}{\mathrm{span}}
\newcommand{\Iso}{\mathrm{Iso}}

\newcommand{\Hei}{\mathbb{H}_3}

 \makeatletter
\tikzset{
  arrow/.style   = {-{Implies}, double equal sign distance, line width=0.65pt},
  iffarrow/.style= {{Implies}-{Implies}, double equal sign distance, line width=0.65pt},
}

\renewcommand{\tocsection}[3]{%
	\indentlabel{\@ifnotempty{#2}{\bfseries\ignorespaces#1 #2\quad}}\bfseries#3}
\renewcommand{\tocsubsection}[3]{%
	\indentlabel{\@ifnotempty{#2}{\ignorespaces#1 #2\quad}}#3}

\def\@tocline#1#2#3#4#5#6#7{\relax
	\ifnum #1>\c@tocdepth 
	\else
	\par \addpenalty\@secpenalty\addvspace{#2}%
	\begingroup \hyphenpenalty\@M
	\@ifempty{#4}{%
		\@tempdima\csname r@tocindent\number#1\endcsname\relax
	}{%
		\@tempdima#4\relax
	}%
	\parindent\z@ \leftskip#3\relax \advance\leftskip\@tempdima\relax
	\rightskip\@pnumwidth plus1em \parfillskip-\@pnumwidth
	#5\leavevmode\hskip-\@tempdima{#6}\nobreak
	\leaders\hbox{$\m@th\mkern \@dotsep mu\hbox{.}\mkern \@dotsep mu$}\hfill
	\nobreak
	\hbox to\@pnumwidth{\@tocpagenum{\ifnum#1=1\bfseries\fi#7}}\par
	\nobreak
	\endgroup
	\fi}
\AtBeginDocument{%
	\expandafter\renewcommand\csname r@tocindent0\endcsname{0pt}
}
\def\l@subsection{\@tocline{2}{0pt}{2.5pc}{5pc}{}}
\makeatother

\DeclareMathOperator{\co}{co}

\DeclareMathOperator{\id}{Id}

\newcommand{\nn}[1]{{\left\vert\kern-0.25ex\left\vert\kern-0.25ex\left\vert #1 
		\right\vert\kern-0.25ex\right\vert\kern-0.25ex\right\vert}}

\renewcommand{\geq}{\geqslant}
\renewcommand{\leq}{\leqslant}

\newcommand{\ep}{\varepsilon}
\newcommand{\supp}{\operatorname{supp}}

\newcommand{\orb}{\operatorname{Orb}}

\newcommand{\eps}{\varepsilon}

\newcommand{\prob}{\operatorname{Prob}}
\newcommand{\meas}{\operatorname{Meas}}

\begin{document}
\title[Strongly Choquet-Deny groups and invariant strictly convex renormings]{Strongly Choquet-Deny groups and invariant strictly convex renormings}

\author[S.~Dantas]{Sheldon Dantas}
\address[S.~Dantas]{Czech Technical University in Prague, FEE, Department of Mathematics, Technická 2, 16627, Prague 6, Czech Republic \newline
\href{https://orcid.org/0000-0001-8117-3760}{ORCID: \texttt{0000-0001-8117-3760}}}
\email{\texttt{sheldon.dantas@fel.cvut.cz}}
\urladdr{www.sheldondantas.com}

\author[M.~Doucha]{Michal Doucha}
\address[M.~Doucha]{Institute of Mathematics, Czech Academy of Sciences, Žitná 25, 115 67 Praha 1, Czechia\newline
\href{https://orcid.org/0000-0003-3675-1378}{ORCID: \texttt{0000-0003-3675-1378}}}
\email{doucha@math.cas.cz}
\urladdr{https://users.math.cas.cz/~doucha/}

\thanks{M. Doucha was supported by the GA\v{C}R project 25-15366S and by the Czech Academy of Sciences (RVO 67985840).}
\begin{abstract}
We introduce and study a strengthening of the notion of Choquet-Deny groups, which we call \emph{strongly Choquet-Deny groups}. We show that a countable group $G$ is strongly Choquet-Deny if and only if for every uniformly bounded representation of $G$ on any separable Banach space $X$, there exists an equivalent $G$-invariant strictly convex norm on $X$. We provide a full classification of finitely generated strongly Choquet-Deny groups and show that a finitely generated group is strongly Choquet-Deny if and only if it is virtually abelian. This fully resolves a question of Mikael de la Salle for finitely generated groups. As a further application, we show that for every uniformly bounded representation of a compact extension of an abelian topological group on any separable Banach space, there exists an equivalent invariant strictly convex norm.
\end{abstract}

\maketitle

\section{Introduction}

The paper connects two areas of mathematics: Choquet-Deny groups on one side and invariant strictly convex renormings of uniformly bounded representations of groups on Banach spaces on the other side.

Choquet-Deny groups themselves have one of the origins in analysis, namely in potential theory and harmonic analysis. In this setting, the focus was on abelian groups, the foundational results being the proof of Choquet and Deny \cite{ChoDe60}, whose name these groups now bear, that all locally compact abelian groups are Choquet-Deny.

Nevertheless, the field has started flourishing upon establishing the connections with boundary theory and random walks in the work of Furstenberg \cite{Fur63}. It was this line of research that has been witnessing the most substantial progress since then (see e.g. \cite{KaVe83,Ka00,Er05} and references therein) and which culminated in the full classification of countable Choquet-Deny groups \cite{FHTV} (building on the earlier result \cite{Ja04}). We refer the reader to several monographs on this topic (\cite{Woe-book,Lal-book,Yad-book}).

Connections with analysis have been maintained throughout the history of the subject, mostly in operator algebras (see e.g. \cite{Izu02} and references therein), ergodic theory (see e.g. \cite{Zim78}), and harmonic analysis (see e.g. \cite{Wi90}). Probably the most direct connection to our paper is Neufang and Jaworski's article on Choquet-Deny groups and group representations in Banach spaces, where the connection between Choquet-Deny groups and strictly convex norms, the main topic of the present paper, appears explicitly (see \cite[Proposition 2.5]{JaNe07}). This connection is the entrance point to the subject for us here.

We were originally motivated by the following question of Mikael de la Salle. Given a (topological) group $G$ acting (continuously) by linear isometries on a Banach space $X$, when does there exist an equivalent strictly convex norm on $X$ that is invariant with respect to the action of $G$? This problem was partially addressed in \cite{FR11,FR13,AFGR19} and the first systematic treatment was initiated in \cite{Do26}. 

Here we present the relation of this problem with a certain strenghtening of Choquet-Deny groups. We introduce a new class of countable groups called here \emph{strongly Choquet-Deny groups} (we refer the reader to Definition~\ref{def:stronglyCD}). The definition is in a similar spirit and given in similar terms as the definition of Choquet-Deny groups and allows one to investigate these groups independently of their motivational origin. Yet it is this relation with de la Salle's problem that we present as one of the two main results of this paper.

\begin{thmx}\label{thm:A}
A countable group $G$ is strongly Choquet-Deny if and only if every uniformly bounded representation of $G$ on a separable Banach space admits an equivalent $G$-invariant strictly convex norm. 
\end{thmx}

This connection made it possible to answer de la Salle’s question for finitely generated groups in full generality, a task that had seemed out of reach when \cite{Do26} was written. Utilizing Theorem~\ref{thm:A} and building on the works of \cite{JaNe07,FHTV}, \cite{AI}, and several other results presented in this article, we obtain a full classification of finitely generated strongly Choquet-Deny groups. While \cite{FHTV}, together with \cite{Ja04}, show that finitely generated Choquet-Deny groups are precisely virtually nilpotent groups, our stronger condition pushes the complexity of strongly Choquet-Deny groups even lower.

\begin{thmx}\label{thm:B}
A finitely generated group is strongly Choquet-Deny if and only if it is virtually abelian.
\end{thmx}

An immediate corollary of Theorems~\ref{thm:A} and~\ref{thm:B} is a complete answer to the question of Mikael de la Salle for finitely generated groups. That is, given a finitely generated group $G$, every separable Banach space equipped with an action of $G$ by uniformly bounded linear operators admits an equivalent $G$-invariant strictly convex norm if and only if $G$ is virtually abelian.

We do not have a classification of countable groups that are not finitely generated and are strongly Choquet-Deny. However, we know that virtually abelian groups are strongly Choquet-Deny, even if they are not finitely generated, and applying \cite{FHTV} we know that groups that have an ICC quotient (i.e. a quotient where every non-trivial element has an infinite conjugacy class) are not strongly Choquet-Deny. The definition of strongly Choquet-Deny groups can be applied also to locally compact groups, however, we do not have an analogue of Theorem~\ref{thm:A} in that setting.

On the other hand, the question whether a $G$-Banach space admits an equivalent $G$-invariant strictly convex norm can be asked for general topological groups. We have the following corollary.

\begin{corollary*} [see Corollary~\ref{cor:cocompactabelian}]
Let $G$ be a topological group that has a closed abelian subgroup $H$ such that $G/H$ is compact. Then every separable $G$-Banach space admits an equivalent $G$-invariant strictly convex norm.
\end{corollary*}

The paper is organized as follows. In Section~\ref{sec:prelim}, we introduce the notation and the definitions of strongly Choquet-Deny and positively strongly Choquet-Deny groups (see Definition \ref{def:stronglyCD}). We also compare these notions with the classical Choquet-Deny groups. Section~\ref{sec:ISCR} provides a new geometric characterization of the existence of invariant strictly convex renormings in terms of extreme points of closed convex hulls of orbits in Theorem \ref{thm:generalequiv}. This is done in the setting of general topological groups acting on Banach spaces and is of independent interest. However, it is immediately applied in the following Section~\ref{sec:CD-ISCR}, where we establish the connection between Choquet-Deny properties and geometric conditions from Section~\ref{sec:ISCR} in Theorem \ref{thm:CD-equiv}, obtaining Theorem \ref{thm:A}. Section~\ref{sec:examples} presents several closure properties of the class of strongly Choquet-Deny groups and it contains our main examples and classification results. In particular, we prove Theorem \ref{thm:B}.

The paper concludes with an appendix, where we construct an action of the lamplighter group on $c$ admitting no equivalent invariant strictly convex norm (see Theorem \ref{theorem:lamplighter-c}). Consequently, $c$ admits no equivalent strictly convex norm invariant under its full linear isometry group as highlighted in Corollary \ref{c-no-equivalent-SC-renorming}), correcting the earlier claim \cite[Theorem 4.8]{Do26} and answering a question on invariant renormings of $c$ from \cite{AFGR19}.

\section{Preliminaries and main definitions}\label{sec:prelim}

Throughout the paper, for convenience, all the Banach spaces and normed spaces are considered to be over the real numbers. Nevertheless, all the results are valid in the same form for complex scalars as well. We refer to Remark~\ref{rem:complex} for more details. However, we emphasize that the function spaces $\ell_1(G)$ and $\ell_\infty(G)$, when $G$ is countable, are always considered real-valued.

We refer the reader who needs additional explanations from Banach space theory (e.g. basics of renorming theory, strictly convex norms, etc.) to \cite{renorming-book} and from group theory (e.g. treatment of the Heisenberg and lamplighter groups, virtually nilpotent groups, etc.) to \cite{DruKabook}.

\subsection{Group actions on Banach spaces}
Although the main results of this paper concern countable groups, several results are stated in the general setting of topological groups continuously acting on Banach spaces. Therefore, we introduce the subject here from this general perspective.

A \emph{representation} of a topological group $G$ on a Banach space $X$ is a continuous action $\alpha: G\times X\to X$ such that for each fixed $g\in G$, the map \[\alpha(g,\cdot):X\to X,\quad x\in X\to \alpha(g,x)\in X,\] is a linear isomorphism of $X$. As usual, if the action $\alpha$ is clear from the context, we suppress it from the notation and write $g\cdot x$ instead of $\alpha(g,x)$.

The representation is called \emph{isometric} if for every $g\in G$, the map $\alpha(g,\cdot):X\to X$ is a linear isometry, in which case $\alpha$ can be identified with a continuous homomorphism from $G$ into $\Iso(X)$, the linear isometry group of $X$ equipped with the strong operator topology.

The representation is called \emph{uniformly bounded} if $\sup_{g\in G} \|\alpha(g,\cdot)\|<\infty$. Every isometric representation is uniformly bounded and conversely, if the representation is uniformly bounded, the Banach space can be renormed to make the representation isometric. Indeed, one can define a new norm defined as follows \[\vertiii{x}:=\sup_{g\in G} \|\alpha(g,x)\|_X,\quad x\in X.\] Uniform boundedness of $\alpha$ guarantees that $\vertiii{\cdot}$ is well-defined and equivalent to the original norm. By definition, $\vertiii{\cdot}$ is $G$-invariant, i.e. $\vertiii{x}=\vertiii{g\cdot x}$, for any $x\in X$ and $g\in G$, which is another way of saying that the representation on $(X,\vertiii{\cdot})$ is isometric.

Notice, in particular, that this observation implies that Theorem~\ref{thm:A} can be equivalently stated with isometric instead of uniformly bounded representations. This prompts us to the following agreement.\medskip

\noindent\textbf{Agreement.} From now on, all the representations of topological groups on Banach spaces are isometric, which are identified with continuous actions of groups on Banach spaces by linear isometries. A Banach space equipped with a continuous action of a topological group $G$ by linear isometries is called a \emph{$G$-Banach space}.\medskip

We also introduce an additional notation related to actions of groups. If a group $G$ acts on a set $X$, which will be in most cases a Banach space, and $x\in X$, resp. $A\subseteq X$, is an element of $X$, resp. a subset of $X$, by $\orb(x)$, resp. $\orb(A)$, we denote the orbit of $x$, resp. the set $\bigcup_{x\in A} \orb(x)$. If $A=\{x_1,\ldots,x_n\}$ is finite, we write $\orb(x_1,\ldots,x_n)$ instead of $\orb(\{x_1,\ldots,x_n\})$. In cases when the acting group $G$ is not clear from the context, we may write $\orb_G(x)$, resp. $\orb_G(A)$, instead of $\orb(x)$, resp. $\orb(A)$.

\subsection{Measures, harmonic functions, and Choquet-Deny groups}
Given a countable group $G$, we denote by $\prob(G)$ the set of probability measures on $G$ which we can identify with $S^+_{\ell_1(G)}$, the positive part of the unit sphere of $\ell_1(G)$. Similarly, $\meas(G)$ denotes the set of signed measures on $G$ of total variation at most $1$, which we can simply identify with $B_{\ell_1(G)}$, the unit ball of $\ell_1(G)$.

Furthermore, given $f\in\ell_\infty(G)$ and $x\in\ell_1(G)$, we define their \emph{convolution} $f\ast x\in\ell_\infty(G)$ by the formula \[f\ast x(g):=\sum_{h\in G} f(gh)x(h^{-1}),\quad g\in G.\]
We also define the $\text{}^*$-operation on $\ell_\infty$ by the formula \[f^*(g):=f(g^{-1}),\quad f\in\ell_\infty(G),\; g\in G.\]
Notice that for $f\in\ell_\infty(G)$ and $x\in\ell_1(G)$, we have $f\ast x^*(g)=\langle g^{-1}\cdot f, x\rangle$, where $\langle \cdot,\cdot\rangle$ is the natural pairing on $\ell_\infty(G)\times\ell_1(G)$.
\begin{definition}
Let $G$ be a countable group and let $\mu$ be a signed measure on $G$. A function $f\in \ell_\infty(G)$ is called \emph{$\mu$-harmonic} if for every $g\in G$ we have \[f(g)=\sum_{h\in G} \mu(h)f(gh).\]
It is called \emph{$\ep$-approximately $\mu$-harmonic}, where $\ep>0$, if for every $g\in G$ we have \[\Big|f(g)-\sum_{h\in G} \mu(h)f(gh) \Big|\leq \|f\|_\infty\cdot\ep.\]
Equivalently, \[\|f-f\ast\mu^*\|_\infty\leq \|f\|_\infty\cdot\varepsilon.\]
\end{definition}

\begin{definition}
A countable group $G$ is \emph{Choquet-Deny} if for every probability measure $\mu\in\prob(G)$, every $f\in\ell_\infty(G)$ that is $\mu$-harmonic is constant on left cosets of $H\leq G$, where $H$ is a subgroup generated by $\supp(\mu):=\{g\in G\colon \mu(g)>0\}$, the support of $\mu$. In particular, if the support of $\mu$ generates $G$, then every bounded $\mu$-harmonic function is constant.
\end{definition}

\begin{remark}
We note that even Choquet-Deny groups are not immune to admitting non-constant approximately harmonic functions. Indeed, consider $\Z$ and a probability measure $\mu\in\prob(\Z)$ defined by $\mu(1)=\mu(-1)=\frac{1}{2}$, and consider any $0<\varepsilon<1$. Then we can define \[f(n):=\begin{cases}
    \min\{n\cdot\varepsilon,1\} & n\geq 0,\\
    \max\{n\cdot\varepsilon,-1\} & n\leq 0.
\end{cases}\]
Clearly, $\|f\|_\infty=1$ and it is non-constant and $\varepsilon$-approximately $\mu$-harmonic.
\end{remark}

\begin{definition}\label{def:stronglyCD}
We call a countable group $G$ \emph{positively strongly Choquet-Deny}, resp. \emph{strongly Choquet-Deny}, if for every sequence $(\mu_n)_{n\in\N}$ of probability measures on $G$, resp. of elements of $\meas(G)$, every decomposition $\mu_n=\frac{\nu_n+\nu'_n}{2}$, where $\nu_n$ and $\nu'_n$ are again probability measures on $G$, resp. elements of $\meas(G)$, for each $n\in\N$, and every function $f\in \ell_\infty(G)$ such that
\begin{enumerate}
    \item $f$ is $\ep_n$-approximately $\mu_n$-harmonic, for each $n\in\N$, where $(\ep_n)_{n\in\N}$ is a sequence of non-negative reals converging to $0$,
    \item \[\Big(\|f\ast (\nu^*_n-(\nu'_n)^*)- f\ast (\nu^*_{n+1}-(\nu'_{n+1})^*)\|_\infty\Big)_{n\in\N}\] is summable (i.e. the sum of the sequence is finite),
\end{enumerate}
there is a subsequence $(\mu_{k_n})_n$ such that \[f\ast (\nu^*_{k_n}-(\nu'_{k_n})^*)\to 0.\]
\end{definition}

\begin{proposition}\label{prop:CDversions}
Let $G$ be a countable group. We have the following series of implications.
\begin{enumerate}
    \item $G$ is strongly Choquet-Deny.
    \par\smallskip\noindent\(\Downarrow\)\par\smallskip

    \item $G$ is positively strongly Choquet-Deny.
    \par\smallskip\noindent\(\Downarrow\)\par\smallskip

    \item $G$ is Choquet-Deny.
\end{enumerate}
\end{proposition}
\begin{proof}
The implication from (1) to (2) is obvious. We prove that (2) implies (3). We shall need the following claim.\medskip

\noindent\textbf{Claim.} For any $\mu\in\prob(G)$,  any $\mu$-harmonic function $f\in\ell_\infty(G)$, and any dyadic tree $(\eta_x)_{x\in 2^{<\N}}\subseteq \prob(G)$ satisfying $\mu=\frac{\eta_1+\eta_2}{2}$ and $\eta_x=\frac{\eta_{x1}+\eta_{x2}}{2}$, for each $x\in 2^{<\N}$, we have that $f$ is $\eta_x$-harmonic, for every $x\in 2^{<\N}$.

We show the claim. Let $(\mu_n)_{n\in\N}$ be the constant sequence, where $\mu_n=\mu$, for each $n$. Set $\nu:=\eta_1$ and $\nu':=\eta_2$. We set $(\nu_n)_{n\in\N}$, resp. $(\nu'_n)_{n\in\N}$, to be the constant sequence, where each element is equal to $\nu$, resp. to $\nu'$. Since $\mu_n=\mu$ and $(\nu^*_n-(\nu'_n)^*)-(\nu^*_{n+1}-(\nu'_{n+1})^*)=(\nu^*-(\nu')^*)-(\nu^*-(\nu')^*)=0$, for each $n\in\N$, we get that $f$ satisfies (1) and (2) of Definition~\ref{def:stronglyCD}. Thus, by definition, there is a subsequence $(k_n)_{n\in\N}$ such that $f\ast (\nu^*_{k_n}-(\nu'_{k_n})^*)\to 0$. However, since the sequences $(\nu_n)_{n\in\N}$, resp. $(\nu'_n)_{n\in\N}$ are constant, we get \[f\ast (\nu^*-(\nu')^*)=0,\;\text{equivalently, }\;f\ast \nu^*=f\ast (\nu')^*.\] It follows 
\[f=f\ast \mu^*=f\ast \frac{\nu^*+(\nu')^*}{2}=f\ast \nu^*=f\ast (\nu')^*,\]
so $f$ is both $\nu$-harmonic and $\nu'$-harmonic, thus $\eta_1$-harmonic and $\eta_2$-harmonic. By continuing the same decomposition, replacing $\mu$ with $\eta_1$ and $\eta_2$ respectively, we inductively prove the claim.\medskip

We are ready to show that $G$ is Choquet-Deny. Fix a probability measure $\mu\in\prob(G)$ and let $f\in\ell_\infty(G)$ be a $\mu$-harmonic function. Pick any $h\in\supp(\mu)$ and let $n\in\N$ be such that $2^{-n}\leq \mu(h)$. Then we can construct a dyadic tree $(\eta_x)_{x\in 2^{<\N}}\subseteq \prob(G)$ satisfying $\mu=\frac{\eta_1+\eta_2}{2}$ and $\eta_x=\frac{\eta_{x1}+\eta_{x2}}{2}$, for each $x\in 2^{<\N}$, such that $\supp(\eta_x)=\{h\}$, for some $x\in 2^{<\N}$ with $|x|=n$. By Claim, $f$ is $\eta_x$-harmonic which means \[f(g)=f\ast \eta^*_x(g)=f(gh),\quad g\in G\] which implies that $f$ is constant on each left coset of $\langle h\rangle\leq G$. Since $h\in\supp(\mu)$ was arbitrary, we get that $f$ is constant on each left coset of $\langle \supp(\mu)\rangle$. Since the $\mu$-harmonic function $f\in\ell_\infty(G)$ was arbitrary, this implies that $G$ is Choquet-Deny.
\end{proof}

The difference between strongly Choquet-Deny groups and positively strongly Choquet-Deny groups is very mild as demonstrated by the following proposition. We postpone its proof to Section~\ref{sec:examples} when we have tools which allow a short proof.

\begin{proposition}\label{prop:SCDandPSCD}
Let $G$ be a countable group. Then $G$ is strongly Choquet-Deny if and only if $G\times\Z_2$ is positively strongly Choquet-Deny.
\end{proposition}

\section{Invariant strictly convex renormings}\label{sec:ISCR}

Here we provide a new characterization of separable $G$-Banach spaces admitting an equivalent $G$-invariant strictly convex norm, where $G$ is a general topological group. This section therefore does not directly concern strongly Choquet-Deny group and can be read independently. However, the result below is then crucial in proving Theorem~\ref{thm:A} from the introduction.

\begin{theorem}\label{thm:generalequiv}
Let $G$ be a topological group and $X$ be a separable $G$-Banach space. The following are equivalent.
\begin{enumerate}
    \item There is a $G$-invariant strictly convex renorming of $X$.
    \item  Every $x\in X$ is an extreme point of $\overline{\co(\orb(x,-x))}$.
    \item There exist a strictly convex normed space $Z$ with a continuous action of $G$ by linear isometries and a bounded injective $G$-equivariant linear operator $T:X\longrightarrow Z$.
\end{enumerate}
 
\end{theorem}

\begin{proof}
Suppose that (1) holds and let $x\in X$ be arbitrary. We can without loss of generality assume that the standard norm on $X$ is already strictly convex. We prove (2) by contraposition. Pick any $x\in X$ and assume that $x=\frac{y+z}{2}$, where $y\neq z\in\overline{\co(\orb(x,-x))}$. Notice that for every $(g_i)_{i=1}^n\subseteq G$ and $(\alpha_i)_{i=1}^n\subseteq \R$, where $n\in\N$ and $\sum_{i=1}^n |\alpha_i|=1$ we have \[\left\|\sum_{i=1}^n \alpha_i g_i\cdot x\right\|\leq \sum_{i=1}^n |\alpha_i|\cdot\|g_i\cdot x\|=\sum_{i=1}^n |\alpha_i|\cdot\|x\|=\|x\|.\] Thus, for every $v\in\co(\orb(x,-x))$ we have $\|v\|\leq\|x\|$, and by continuity, in turn for every $w\in\overline{\co(\orb(x,-x))}$ we have $\|w\|\leq\|x\|$. By strict convexity of $\|\cdot\|$ it follows that \[\|x\|=\left\|\frac{y+z}{2}\right\|<\frac{2\|x\|}{2}=\|x\|\] which is a contradiction.\bigskip

Suppose now that (2) holds. For every $f\in B_{X^*}$, the dual unit ball of $X$, we define \[\|x\|_f:=\sup_{g\in G} |f(g\cdot x)|,\quad x\in X.\] Since for every fixed $g\in G$, $|f(g\cdot z)|$, $z\in X$, defines a seminorm on $X$ dominated by $\|\cdot\|$, we get $\|\cdot\|_f$ is a $G$-invariant seminorm dominated by $\|\cdot\|$; in particular, it is continuous. 

For every $f\in B_{X^*}$, define
\[ U_f:=\left\{(x,y)\in X^2\setminus\Delta\colon 
    \max\{\|x\|_f,\|y\|_f\}>\left\|\frac{x+y}{2}\right\|_f\right\},\] where $\Delta:=\{(v,v)\colon v\in X\}$. By continuity of $\|\cdot\|_f$, $U_f$ is an open subset of $X^2\setminus\Delta$.\medskip

\noindent\textbf{Claim.} $X^2\setminus\Delta=\bigcup_{f\in B_{X^*}} U_f$.

We prove the claim. Pick any $(x,y)\in X^2\setminus\Delta$ and set $z:=\frac{x+y}{2}$. By the assumption, since $z$ is an extreme point of $\overline{\co(\orb(z,-z))}$, either $x\notin\overline{\co(\orb(z,-z))}$ or $y\notin\overline{\co(\orb(z,-z))}$. Let us assume the former. By the Hahn-Banach separation theorem, there exists $f\in B_{X^*}$ such that \[f(x)>\sup\{f(v)\colon v\in \overline{\co(\orb(z,-z))}\}=\sup\{f(v)\colon v\in\orb(z,-z)\}.\] Therefore, we have \[\|x\|_f\geq |f(x)|\geq f(x)>\sup\{\|v\|_f\colon v\in\orb(z,-z)\}=\|z\|_f,\] where the strict inequality follows from the fact that $\orb(z,-z)$ is a symmetric set.

If $y\notin\overline{\co(\orb(z,-z))}$, then we can analogously get $f\in B_{X^*}$ such that $\|y\|_f>\|z\|_f$. In either case, for every $(x,y)\in X^2\setminus\Delta$ there is $f\in B_{X^*}$ such that \[\max\{\|x\|_f,\|y\|_f\}>\left\|\frac{x+y}{2}\right\|_f,\] which proves the claim.\medskip

Since $X$ is separable, $X^2\setminus\Delta$ is second-countable and so it is Lindel\"of. Therefore $X^2\setminus\Delta=\bigcup_{f\in B_{X^*}} U_f$ has a countable subcover $\bigcup_{n\in\N} U_{f_n}$ given by functionals $(f_n)_{n\in\N}\subseteq B_{X^*}$. For simplicity, we shall denote each seminorm $\|\cdot\|_{f_n}$ by $\|\cdot\|_n$.

We define a new seminorm $\vertiii{\cdot}$ on $X$. We set \[\vertiii{x}:=\left(\sum_{n=1}^\infty \frac{\|x\|_n^2}{2^n}\right)^{1/2},\quad x\in X.\] Since each $\|\cdot\|_n$ is $G$-invariant, so is $\vertiii{\cdot}$. Moreover, as $\|\cdot\|_n\leq \|\cdot\|$, for all $n\in\N$, we get that $\vertiii{\cdot}\leq \|\cdot\|$. We check that $\vertiii{\cdot}$ is a norm. Pick any $x\neq 0$. There is $n\in\N$ such that $(x,0)\in U_{f_n}$, thus $\|x\|_n>\|\frac{x}{2}\|_n$. In particular, $\|x\|_n>0$, and so $\vertiii{x}>0$. Finally, we check that $\vertiii{\cdot}$ is strictly convex. Pick $(x,y)\in X^2\setminus\Delta$ such that $\vertiii{x}=\vertiii{y}=1$ and suppose, in order to reach a contradiction, that $\vertiii{x+y}=2$. Then,
\[\begin{split}
2=\left(\sum_{n=1}^{\infty}\frac{\|x+y\|_n^2}{2^n}\right)^{1/2}
&\leq\left(\sum_{n=1}^{\infty}\frac{(\|x\|_n+\|y\|_n)^2}{2^n}\right)^{1/2}\\
&\leq\left(\sum_{n=1}^{\infty}\frac{\|x\|_n^2}{2^n}\right)^{1/2}
+\left(\sum_{n=1}^{\infty}\frac{\|y\|_n^2}{2^n}\right)^{1/2}
=\vertiii{x}+\vertiii{y}=2.
\end{split}\]
Since all the inequalities are equalities, the equality in the Minkowski inequality implies that there is a scalar $c>0$ such that $\|x\|_n=c\|y\|_n$ for every $n\in\N$. Since $\vertiii{x}=\vertiii{y}=1$, we have $c=1$. Thus, $  \|x\|_n=\|y\|_n$ for every $n\in\N$. However, there is $n\in\N$ such that $(x,y)\in U_{f_n}$, so $\max\{\|x\|_n,\|y\|_n\}>\|\frac{x+y}{2}\|_n$, which contradicts that $\|x\|_n=\|y\|_n$.

It follows that $(X,\vertiii{\cdot})$ is a strictly convex normed space equipped with a continuous action of $G$ preserving the norm. Since $\vertiii{\cdot}\leq\|\cdot\|$, the identity map between $(X,\|\cdot\|)$ and $(X,\vertiii{\cdot})$ is an injective non-expansive $G$-equivariant linear operator. This proves (3).\bigskip

Assume (3). The proof is the same as the proof of \cite[Proposition 3.1]{Do26}, where the only difference was that $Z$ was a Banach space instead of just being a normed space. This difference is immaterial. For the convenience of the reader, we sketch the proof and we refer to \cite{Do26} for more details. We define a new norm $\vertiii{\cdot}$ as follows. We set \[\vertiii{x}:=\|x\|+\|T(x)\|_Z,\quad x\in X.\] Clearly, $\vertiii{\cdot}$ is an equivalent norm on $X$. Since $\|\cdot\|$ and $\|\cdot\|_Z$ are $G$-invariant and $T$ is $G$-equivariant, $\vertiii{\cdot}$ is $G$-invariant. Finally, it is strictly convex as a sum of two norms, one of them being strictly convex. This proves (1).
\end{proof}

We state and prove one additional lemma that is both of independent interest as well as it is applied later.

\begin{lemma}\label{lem:cocompactsubgroup} Let $G$ be a topological group and $H$ be a closed subgroup such that $G/H$ is compact. Furthermore, let $X$ be a $G$-Banach space and suppose that the subaction of $H$ on $X$ admits a strictly convex $H$-invariant renorming. Then $X$ admits a strictly convex $G$-invariant renorming.
\end{lemma}

\begin{proof} Let $\|\cdot\|$ be an equivalent strictly convex $H$-invariant norm on $X$. For every $x \in X$, define $\vertiii{x} := \sup_{g \in G} \|g \cdot x\|$, which is an equivalent $G$-invariant norm. We will show that this supremum is attained. For each $x \in X$, the function $gH \mapsto \|g^{-1} \cdot x\|$ from $G/H$ into $\R$ is well-defined since 
\begin{equation*}
    \|(gh)^{-1} \cdot x\| = \|h^{-1} g^{-1} \cdot x \| = \|g^{-1} \cdot x\| 
\end{equation*}
for every $h \in H$. Moreover, since $g \mapsto \|g^{-1} \cdot x\|$ is continuous, the above map is also continuous. Now, using the assumption that $G/H$ is compact, we obtain that 
\begin{equation*}
\vertiii{x} = \sup_{g \in G} \|g^{-1} \cdot x\| = \max_{gH \in G/H} \| g^{-1} \cdot x\|     
\end{equation*}
for every $x \in X$. By \cite[Proposition 4.3]{Do26}, the norm $\vertiii{\cdot}$ is strictly convex.
\end{proof}

\section{Choquet-Deny properties and invariant renormings}\label{sec:CD-ISCR}

The following theorem establishes the connections between various Choquet-Deny properties of groups and extreme points of closed convex hull of the orbits of actions on Banach spaces. This together with the result from the previous section will prove Theorem~\ref{thm:A}.

\begin{theorem}\label{thm:CD-equiv}
Let $G$ be a countable group. The following are equivalent.

\begin{enumerate}
    \item For every separable $G$-Banach space $X$, every $x\in X$ is an extreme point of $\overline{\co(\orb(x,-x))}$.
    \item $G$ is strongly Choquet-Deny.
\end{enumerate}

Analogously, we have an equivalence between the following two statements,
\begin{enumerate}[label={(\arabic*')}]
    \item For every separable $G$-Banach space $X$, every $x\in X$ is an extreme point of $\overline{\co(\orb(x))}$.
    \item $G$ is positively strongly Choquet-Deny.
\end{enumerate}

as well as between the following two statements.
\begin{enumerate}[label={(\arabic*'')}]
    \item For every separable $G$-Banach space $X$, no $x\in X$ can be written as a non-trivial, possibly infinite, convex combination of elements from $\orb(x)$; that is, if $x=\sum_{n=1}^\infty \alpha_n g_n\cdot x$, where $(g_n)_n\subseteq G$ and $(\alpha_n)_n\subseteq \left(0,\infty\right)$ with $\sum_{n=1}^{\infty} \alpha_n=1$, then for all $n$, $g_n\cdot x=x$.
    \item $G$ is Choquet-Deny.
\end{enumerate}
\end{theorem}

\begin{proof} We shall show the equivalence between (1) and (2). The equivalences between (1') and (2'), resp. between (1'') and (2'') can be proved analogously.

Suppose that $G$ is not strongly Choquet-Deny. Then, there exist $f\in\ell_\infty(G)$, a sequence $(\ep_n)_{n\in \N}$ of positive reals tending to $0$ and a sequence $(\mu_n)_{n\in\N}\subseteq \meas(G)$  along with decompositions $\mu_n=\frac{\nu_n+\nu'_n}{2}$, where $\nu_n,\nu'_n\in \meas(G)$ for each $n$ satisfying the following conditions. First, $f$ is $\ep_n$-approximately $\mu_n$-harmonic, for each $n$, and $\Big(\|f\ast (\nu^*_n-(\nu'_n)^*)- f\ast (\nu^*_{n+1}-(\nu'_{n+1})^*)\|_\infty\Big)_{n\in\N}$ is summable. Second, for every subsequence $(\mu_{k_n})_n$, $f\ast (\nu^*_{k_n}-(\nu'_{k_n})^*)$ does not converge to $0$.
Finally, without loss of generality, we assume that $\|f\|_\infty=1$.

For each $n\in\N$ we define $M_n:\ell_\infty(G)\to\R$ as follows. \[M_n(\phi):=\sup_{g\in G} \left|\phi(g)-\sum_{h\in G} \mu_n(h)\phi(gh)\right|,\quad \phi\in\ell_\infty(G).\] Notice that for each $\phi\in\ell_\infty(G)$, $\frac{M_n(\phi)}{\|\phi\|_\infty}$ is the least constant $K$ such that $\phi$ is $K$-approximately $\mu_n$-harmonic.

Similarly, for each $n\in\N$ we define $N_n:\ell_\infty(G)\to\R$ by the following formula. \[N_n(\phi):=\|\phi\ast (\nu^*_n-(\nu'_n)^*)- \phi\ast (\nu^*_{n+1}-(\nu'_{n+1})^*)\|_\infty,\quad \phi\in\ell_\infty(G),\] and for each $n\in\N$, set $\delta_n:=\max\{N_n(f),2^{-n}\}$.

Let us define the following function $\vertiii{\cdot}^*:\ell_\infty(G)\to\left[0,\infty\right]$. We set \[\vertiii{\phi}^*:=\max\left\{\|\phi\|_\infty,\sup_{n\in\N} \frac{M_n(\phi)}{\ep_n}, \sup_{n\in\N} \frac{N_n(\phi)}{\delta_n}\right\},\quad \phi\in\ell_\infty(G).\] Notice that it vanishes only at the zero function and it is homogeneous and subadditive (under the agreement that $k\cdot\infty=\infty$, for $k\neq 0$, $0\cdot\infty=0$, and $k+\infty\leq \infty$, for all $k\in\R\cup\{\infty\}$). Moreover, since by the assumption $\|f\|_\infty=1$, we get by definition that $\vertiii{f}^*=1$.

Now set 
\[X^*:=\big\{\phi\in\ell_\infty(G) \colon \vertiii{\phi}^*<\infty\big\}\]
and equip it with the norm $\vertiii{\cdot}^*$.\medskip

\noindent\textbf{Claim 1.} $X^*$ and $\vertiii{\cdot}^*$ are invariant under the action of $G$.

The claim immediately follows from the definition of $\vertiii{\cdot}^*$ and the fact that the original norm $\|\cdot\|_\infty$ as well as the maps $M_n$ and $N_n$, for all $n$, are all $G$-invariant. Indeed, for the norm we have already noticed it and for $M_n$ it is obvious from the definition. For $N_n$, it follows from its definition and again the fact that the norm $\|\cdot\|_\infty$ is $G$-invariant.\medskip

Finally, let us define the following seminorm $\vertiii{\cdot}$ on $\ell_1(G)$. We set \[\vertiii{x}:=\sup\left\{\frac{|\phi(x)|}{\vertiii{\phi}^*}\colon \phi\in X^*\setminus\{0\}\right\},\quad x\in\ell_1(G).\]\medskip

\noindent\textbf{Claim 2.} For every $x\in\ell_1(G)$, $\vertiii{x}\leq \|x\|_1$.

This follows from the definition of $\vertiii{\cdot}$ and the fact that for $\phi\in X^*$, we have $\vertiii{\phi}^*\geq \|\phi\|_\infty$.\medskip

We denote by $X$ the Banach space obtained from $(\ell_1(G),\vertiii{\cdot})$ after taking the quotient by the subspace $\{x\in\ell_1(G)\colon \vertiii{x}=0\}$ and taking the completion.\medskip

\noindent\textbf{Claim 3.} $G$ acts on $X$ by linear isometries.

It suffices to realize that the seminorm $\vertiii{\cdot}$ is $G$-invariant with respect to the standard action of $G$ on $\ell_1(G)$. Then it means that $G$ acts by linear isometries on the seminormed space $(\ell_1(G),\vertiii{\cdot})$ and so it also acts by linear isometries on the corresponding quotient and completion. However, $\vertiii{\cdot}$ is $G$-invariant by definition since $\vertiii{\cdot}^*$ is $G$-invariant.\medskip

For each $g\in G$, denote by $\delta_g\in\ell_1(G)$ the indicator function of the singleton $\{g\}$.\medskip

\noindent\textbf{Claim 4.} For each $g\in G$, $\vertiii{\delta_g}=1$.

Since by Claim 2, $\vertiii{\delta_g}\leq \|\delta_g\|_1=1$, it suffices to check that $\vertiii{\delta_g}\geq 1$. Since $\|f\|_\infty=1$, for every $n\in\N$ there is $g_n\in G$ such that $|f(g_n)|>1-1/n$, thus $|gg_n^{-1}\cdot f(g)|>1-1/n.$ Moreover, since $\vertiii{f}^*=1$ we get 
\[\vertiii{\delta_g}\geq \sup_n \frac{|(gg_n^{-1}\cdot f)(g)|}{\vertiii{gg_n^{-1}\cdot f}^*}=\sup_n |(gg_n^{-1}\cdot f)(g)|=1.\]\medskip

\noindent\textbf{Claim 5.} For every $g\in G$, $\lim_{n\to\infty} g\cdot \mu_n=\delta_g$, in the norm topology of $X$.

It suffices to check that $(\mu_n)_n$ seen as a sequence in $X$ converges to $\delta_{1_G}$ (here we identify each $\mu_n$, which is an element of $\ell_1(G)$, with its equivalence class inside $X$). In fact, we claim that for each $n\in\N$, we have $\vertiii{\mu_n-\delta_{1_G}}\leq \ep_n$. Verifying this will prove Claim 5 since $\ep_n\to 0$. Let $\phi\in X^*$ be such that $\vertiii{\phi}^*=1$. Then \[\big|\langle \mu_n-\delta_{1_G}, \phi\rangle\big|= \left|\left(\sum_{h\in G} \mu_n(h)\phi(h)\right)-\phi(1_G)\right|\leq M_n(\phi)\leq \ep_n\cdot\vertiii{\phi}^*=\ep_n.\] Since $\phi\in X^*$ with $\vertiii{\phi}^*=1$ was arbitrary, we get the desired $\vertiii{\mu_n-\delta_{1_G}}\leq \ep_n$.\medskip

\noindent\textbf{Claim 6.} $(\nu_n-\nu'_n)_{n\in\N}$ converges to a non-zero element of $X$.

First we observe that $(\nu_n-\nu'_n)_{n\in\N}$ converges. Indeed, analogously as in the proof of Claim 5 we pick an arbitrary $\phi\in X^*$ with $\vertiii{\phi}^*=1$ and any $n\in\N$ to check that \[\begin{split}\big|\langle (\nu_n-\nu'_n)-(\nu_{n+1}-\nu'_{n+1}), \phi\rangle\big|&=\big|\phi\ast (\nu^*_n-(\nu'_n)^*)-\phi\ast (\nu^*_{n+1}-(\nu'_{n+1})^*)\big(1_G\big)\big|\\ &\leq N_n(\phi)\leq \delta_n\cdot \vertiii{\phi}^*=\delta_n,\end{split}\] verifying that $\vertiii{(\nu_n-\nu'_n)-(\nu_{n+1}-\nu'_{n+1})}\leq \delta_n$ as $\phi\in X^*$, satisfying $\vertiii{\phi}^* = 1$, was arbitrary. Since $\sum_{n=1}^\infty \delta_n<\infty$, we get that $(\nu_n-\nu'_n)_{n\in\N}$ is a Cauchy sequence, thus it converges. For the sake of contradiction, suppose that it converges to $0$. Thus it also weakly converges to $0$, so we must have \[\lim_{n\to\infty} \langle \nu_n-\nu'_n, g^{-1}\cdot f\rangle=\lim_{n\to\infty} f\ast (\nu^*_n-(\nu'_n)^*)\big(g\big)=0,\] for all $g\in G$, which contradicts the assumption that $f\ast (\nu^*_{k_n}-(\nu'_{k_n})^*)$ does not converge to zero, i.e. for some $g\in G$, $\langle \nu_{k_n}-\nu'_{k_n}, g^{-1}\cdot f\rangle$ does not converge to zero.\medskip

Finally, since $\nu_n=\mu_n+\frac{\nu_n-\nu'_n}{2}$ and $\nu'_n=\mu_n-\frac{\nu_n-\nu'_n}{2}$, and both $(\mu_n)_n$ and $(\nu_n-\nu'_n)_n$ converge, we get that $(\nu_n)_n$ and $(\nu'_n)_n$ converge as well to some $\nu\in X$, resp. $\nu'\in X$. We have \[\delta_{1_G}=\lim_{n\to\infty} \mu_n=\lim_{n\to\infty} \frac{\nu_n+\nu'_n}{2}=\lim_{n\to\infty} \frac{\nu_n}{2}+\lim_{n\to\infty} \frac{\nu'_n}{2}=\frac{\nu+\nu'}{2}.\] Moreover, \[0\neq \lim_{n\to\infty} (\nu_n-\nu'_n)=\lim_{n\to\infty} \nu_n-\lim_{n\to\infty} \nu'_n=\nu-\nu',\] so $\nu\neq \nu'$. Since for each $n\in\N$ we have $\nu_n,\nu'_n\in \overline{\co(\orb(\delta_{1_G},-\delta_{1_G}))}$ we get that $\nu,\nu'\in\overline{\co(\orb(\delta_{1_G},-\delta_{1_G}))}$ and so $\delta_{1_G}$ is not an extreme point of $\overline{\co(\orb(\delta_{1_G},-\delta_{1_G}))}$.\bigskip

Suppose that there exist a separable $G$-Banach space $X$ and an element $x\in X$ which is not an extreme point of $\overline{\co(\orb(x,-x))}$. Thus there are $u\neq v\in \overline{\co(\orb(x,-x))}$ such that $\frac{u+v}{2}=x$. Without loss of generality, we may and will assume that $\|x\|=1$. Let $(u_n)_{n\in\N}\subseteq \co(\orb(x,-x))$, resp. $(v_n)_{n\in\N}\subseteq \co(\orb(x,-x))$ be sequences of elements from the convex hull of the orbit of $\{x,-x\}$ that converge to $u$ and $v$ respectively. For each $n\in\N$, let $\nu_n$, resp. $\nu'_n$ be elements from $\meas(G)$ such that \[u_n=\sum_{g\in G} \nu_n(g)g\cdot x,\quad v_n=\sum_{g\in G} \nu'_n(g) g\cdot x.\] We set $\mu_n:=\frac{\nu_n+\nu'_n}{2}$, for each $n$.

Now pick any $\phi\in X^*$ and let us define $f_\phi\in\ell_\infty(G)$ by the formula \[f_\phi(g):=\phi(g\cdot x), \quad g\in G.\] Clearly, $\|f_\phi\|_\infty\leq \|\phi\|_{X^*}$. In fact, by a convexity argument we have 
\begin{equation}\label{eq:sumof2fromorbit}
\|f_\phi\|_\infty=\sup_{g\in G} |\phi(g\cdot x)|=\sup_{y\in \overline{\co(\orb(x))}} |\phi(y)|.
\end{equation}
\medskip

\noindent\textbf{Claim 7.} Each such $f_\phi$ is $\ep_n$-approximately $\mu_n$-harmonic, for every $n$, where $(\ep_n)_{n\in\N}$ is a sequence of non-negative reals tending to $0$.

Fix $n \in \N$ and $g \in G$. We have that 
\begin{equation*}
    \|f_{\phi} - f_{\phi} \ast \mu_n^*\|_{\infty} = \sup_{g \in G} \left| \phi \left(g \cdot \left( x - \frac{u_n + v_n}{2} \right) \right) \right| \leq \|\phi\| \left\| x - \frac{u_n + v_n}{2} \right\|.
\end{equation*}
For $f_{\phi} \not=0$, we may take 
\begin{equation*}
    \eps_n := \frac{\|\phi\|}{\|f_{\phi}\|_{\infty}} \left\| x - \frac{u_n + v_n}{2} \right\|
\end{equation*}
which is non-negative and tends to zero (if $f_{\phi} = 0$, then any positive sequence tending to zero will do the job although, in Claim 9 below, we will have necessarily $f_{\phi}\not= 0$). Since $\frac{u_n + v_n}{2} \rightarrow \frac{u + v}{2} = x$, we get that $f_{\phi}$ is an $\eps_n$-approximately $\mu_n$-harmonic for every $n$.
\medskip

\noindent\textbf{Claim 8.} For each such $f_\phi$, there is a subsequence $(k_n)_{n\in\N}$ such that \[\Big(\|f_\phi\ast (\nu^*_{k_n}-(\nu'_{k_n})^*)-f_\phi\ast (\nu^*_{k_{n+1}}-(\nu'_{k_{n+1}})^*)\|_\infty\Big)_{n\in\N}\] is summable.

For each $n$, we have \[\begin{split}&\|f_\phi\ast (\nu^*_n-(\nu'_n)^*)-f_\phi\ast (\nu^*_{n+1}-(\nu'_{n+1})^*)\|_\infty\\ &=\sup_{g\in G} \big|f_\phi\ast (\nu^*_n-(\nu'_n)^*)-f_\phi\ast (\nu^*_{n+1}-(\nu'_{n+1})^*)\big(g\big)\big|\\ &=\sup_{g\in G} \big|g^{-1}\cdot \phi\big((u_n-v_n)-(u_{n+1}-v_{n+1})\big)\big|\\ &\leq \|\phi\|_{X^*}\cdot \|(u_n-v_n)-(u_{n+1}-v_{n+1})\|.\end{split}\] Since $u_n-v_n\to u-v$, thus $(u_n-v_n)_{n\in\N}$ is a Cauchy sequence in $X$, there exists a subsequence $(k_n)_{n\in\N}$ such that the sequence $(\|(u_{k_n}-v_{k_n})-(u_{k_{n+1}}-v_{k_{n+1}})\|)_{n\in\N}$ is summable. By the inequality above, this proves the claim.\medskip

\noindent\textbf{Claim 9.} $G$ is not strongly Choquet-Deny.

Let $\phi\in X^*$ be such that $\phi(u-v)\neq 0$. By Claim 7, $f_\phi$ is $\ep_n$-approximately $\mu_n$-harmonic for some sequence $(\ep_n)_{n\in\N}$ tending to $0$ and by Claim 8, without loss of generality, by passing to a subsequence if necessary, we assume that also $\Big(\|f_\phi\ast (\nu^*_n-(\nu'_n)^*)-f_\phi\ast (\nu^*_{n+1}-(\nu'_{n+1})^*)\|_\infty\Big)_{n\in\N}$ is summable, so (1) and (2) of Definition~\ref{def:stronglyCD} are satisfied for $f_\phi$ and $(\mu_n=\frac{\nu_n+\nu'_n}{2})_n$. However, since \[\lim_{n\to\infty} f_\phi\ast (\nu^*_n-(\nu'_n)^*)(1_G)=\lim_{n\to\infty} \phi\big(u_n-v_n\big)=\phi(u-v)\neq 0\] we get that for no subsequence $(k_n)_{n\in\N}$ we have $f_{\phi} \ast (\nu^*_{k_n}-(\nu'_{k_n})^*)\to 0$ showing that $f_\phi$ along with $(\mu_n= \frac{\nu_n+\nu'_n}{2})_n$ contradicts that $G$ is strongly Choquet-Deny.
\end{proof}

The following is therefore an immediate corollary of Theorems~\ref{thm:generalequiv} and~\ref{thm:CD-equiv}.

\begin{theorem}
Let $G$ be a countable group. Then the following are equivalent.

\begin{enumerate}
    \item $G$ is strongly Choquet-Deny.
    \item Every separable $G$-Banach space admits an equivalent strictly convex $G$-invariant norm.
\end{enumerate}
\end{theorem}

We conclude with a remark explaining the case of complex scalars.

\begin{remark}\label{rem:complex}
All the results in this and the preceding sections are valid also in the case of complex scalars. This is based on the following two observations. We fix a topological group $G$.
\begin{enumerate}
    \item Suppose that a complex $G$-Banach space $X$ admits an equivalent $G$-invariant strictly convex real norm. Then it admits such a complex norm as well.
    \item Suppose that a real $G$-Banach space $X$ does not admit an equivalent $G$-invariant strictly convex norm. Then neither does its complexification, which is a complex $G$-Banach space.
\end{enumerate}

For (1), if $\|\cdot\|$ is an equivalent $G$-invariant strictly convex real norm on $X$, then \[\vertiii{x}:=\max_{z\in\T} \|z\cdot x\|,\quad x\in X,\] where $\T:=\{z\in\C\colon |z|=1\}$, is an equivalent $G$-invariant strictly convex complex norm on $X$. We leave the details to the reader.

For (2), first notice that the complexification is naturally a $G$-Banach space as well by considering the sum of the original action on $X\oplus X$. Second, observe that the embedding of $X$ into its complexification is a $G$-equivariant linear isometry, when the range is seen as a real space. Thus, if the complexification admits an equivalent $G$-invariant strictly convex norm, then so does its subspace $X$.
\end{remark}

\section{Strongly Choquet-Deny groups and their classification}\label{sec:examples}

In this section, we derive several closure properties of the class of strongly Choquet-Deny groups. We prove that abelian groups are strongly Choquet-Deny and that the Heisenberg group is not. This will culminate with the classification of finitely generated strongly Choquet-Deny groups and hereby with the proof of Theorem~\ref{thm:B}.

\begin{proposition}\label{prop:closureofCD}
The class of strongly Choquet-Deny groups is closed under
\begin{enumerate}
    \item taking subgroups;
    \item taking finite-index overgroups.
\end{enumerate}
\end{proposition}
\begin{proof}
We show (1). Let $G$ be a strongly Choquet-Deny countable group and $H\leq G$ be a subgroup. Let $(g_n)_{n\in\N}\subseteq G$ be an enumeration of left-coset representatives, where $g_1=1_G$. We may assume that it is infinite, the finite case is easier. Let $X$ be a separable $H$-Banach space. We set \[Y:=\left(\bigoplus_{n\in \N} X_n\right)_{\ell_2}\] to be the $\ell_2$-sum of copies of $(X_n)_{n\in\N}$, where each $X_n$ is a copy of $X$. In particular, each $X_n$ carries a linear isometric action of $H$. We define the induced action of $G$ on $Y$. First, notice that we have an action of $G$ on $\N$ defined by $g\cdot n=m$ if $gg_n\in g_mH$. Second, we define an $H$-valued \emph{cocycle} $\sigma: G\times\N\to H$ by setting $\sigma(g,n)$ to be the unique $h\in H$ such that $gg_n=g_{g\cdot n}h$. We are ready to define $G\curvearrowright Y$ as follows. We set \[g\cdot (x_n)_n:=(\sigma(g,g^{-1}\cdot n)x_{g^{-1}\cdot n})_{n\in\N},\quad g\in G,\;(x_n)_{n\in\N}\in Y.\] This is a linear action by isometries. Since for every $h\in H$ we have $hg_1=h=g_1h$ we have that the action of $H$ on $X_1$ coincides with the action of $H$ on $X$, i.e. there is an $H$-equivariant linear isometry $T:X\to X_1\subseteq Y$. Since by the assumption $G$ is strongly Choquet-Deny, there is an equivalent $G$-invariant strictly convex norm on $Y$, and the $H$-equivariant linear map $T$ pushes the strictly convex norm from the subspace $X_1$ to $X$. Therefore, $X$ admits an equivalent $H$-invariant strictly convex norm. Since $X$ was arbitrary, it follows, by \Cref{thm:CD-equiv}, that $H$ is strongly Choquet-Deny.\medskip

To prove (2), let $G$ be a countable group and $H$ be a finite-index subgroup that is strongly Choquet-Deny. Let $X$ be a separable $G$-Banach space. By \Cref{thm:CD-equiv}, there is an equivalent $H$-invariant strictly convex norm on $X$ and since $G/H$ is finite, i.e. compact, by Lemma~\ref{lem:cocompactsubgroup}, there is also an equivalent $G$-invariant strictly convex norm on $X$. Again, since $X$ was arbitrary, it follows, by \Cref{thm:CD-equiv}, that $G$ is strongly Choquet-Deny.
\end{proof}

We are now almost ready to prove Proposition~\ref{prop:SCDandPSCD} connecting strongly Choquet-Deny and positively strongly Choquet-Deny groups. We just need a simple lemma.

\begin{lemma} \label{extend-to-G-Z2}
Every action of any topological group $G$ on any Banach space $X$ by linear isometries extends to an action of $G\times \Z_2$ on $X$, where the non-trivial element of $\Z_2$ acts as $-\id$.
\end{lemma}
\begin{proof}
This is immediate since $-\id$ is a central element of $\Iso(X)$, i.e. it commutes with all linear isometries on $X$.
\end{proof}

\begin{proof}[Proof of Proposition~\ref{prop:SCDandPSCD}]
Suppose first that $G$ is strongly Choquet-Deny. By Proposition~\ref{prop:closureofCD}~(2), $G\times\Z_2$ is strongly Choquet-Deny as well and thus, by Proposition~\ref{prop:CDversions}, $G\times\Z_2$ is positively strongly Choquet-Deny.

Conversely, suppose that $G\times \Z_2$ is positively strongly Choquet-Deny. Let $X$ be any separable $G$-Banach space. By Lemma \ref{extend-to-G-Z2}, we can extend the action of $G$ on $X$ to $G \times \Z_2$ so that the non-trivial element of $\Z_2$ acts as $-\id$. It follows that for every $x\in X$ we have \[
    \overline{\co(\orb_{G \times \Z_2} (x))} = \overline{\co(\orb_G(x, -x))}.\]
Thus, for any $x\in X$, since $G\times\Z_2$ is positively strongly Choquet-Deny, the equivalence between (1') and (2') in Theorem~\ref{thm:CD-equiv} gives that $x$ is an extreme point of shows that every $x \in X$ is an extreme point of $\overline{\co(\orb_{G \times \Z_2} (x))} = \overline{\co(\orb_G(x, -x))}$, which then by the equivalence between (1) and (2) in Theorem~\ref{thm:CD-equiv} gives that $G$ is strongly Choquet-Deny since $X$ was arbitrary.
\end{proof}

\begin{theorem}\label{thm:abelianCD}
Abelian groups are strongly Choquet-Deny.
\end{theorem}
\begin{proof}
 Fix a countable abelian group $G$ and a sequence of signed measures $(\mu_n)_{n\in\N}\subseteq \meas(G)$, decompositions $\mu_n=\frac{\nu_n+\nu'_n}{2}$, for every $n$, and a function $f\in\ell_\infty(G)$ satisfying (1) and (2) of Definition~\ref{def:stronglyCD}. Without loss of generality, assume that $\|f\|_\infty=1$. In order to reach a contradiction, suppose that for no subsequence $(k_n)_{n\in\N}$ we have $f\ast (\nu^*_{k_n}-(\nu'_{k_n})^*)\to 0$.

 Set $X:=\overline{\Span(\orb(f,-f))}=\overline{\Span(\orb(f))}\subseteq \ell_\infty(G)$. By (1), we have that $\|f-f\ast \mu^*_n\|_\infty\leq\ep_n$ and since being $\ep_n$-approximately $\mu_n$-harmonic is a $G$-invariant condition, we get that $\|g\cdot f-g\cdot f\ast \mu^*_n\|_\infty\leq \ep_n$, for all $g\in G$. It follows that for every $x\in\Span(\orb(f))$ we have $\|x-x\ast\mu^*_n\|_\infty\to 0$. 
 
 Since $G$ is abelian, for every $x\in X$ we have $x\ast\mu^*_n=\mu^*_n\ast x$ implying that $x\ast\mu^*_n\in X$, for all $x\in X$ and $n\in\N$. Therefore \[T_n(x):=x\ast\mu^*_n,\quad x\in X\] defines a bounded linear operator $T_n:X\to X$, for each $n\in\N$, with $\|T_n\|\leq 1$ since $\|\mu_n\|_1\leq 1$. We have
 \begin{equation}\label{eq:T_n-to-Id}
  \|T_n(x)-x\|_\infty\to 0
 \end{equation}
for all $x\in\Span(\orb(f))$. By continuity, \eqref{eq:T_n-to-Id} extends to all $x\in X$, so the sequence $(T_n)_{n\in\N}$ converges in the strong operator topology to the identity operator $\id_X$.
 
Similarly, for every $x\in \Span(\orb(f))$ we get 
 \begin{equation}\label{eq:summable}
 \Big(\|x\ast (\nu^*_n-(\nu'_n)^*)-x\ast (\nu^*_{n+1}-(\nu'_{n+1})^*)\|\Big)_{n\in\N}\;\text{ is summable,} 
 \end{equation}
since \eqref{eq:summable} holds by definition for $f$, by $G$-invariance of \eqref{eq:summable} for $g\cdot f$, for all $g\in G$, and \eqref{eq:summable} is clearly closed under linear combination.

Therefore, the formula \[U_n(x):=x\ast \Big(\frac{\nu_n^*-(\nu'_n)^*}{2}\Big),\quad x\in X\] again defines a linear operator $U_n:X\to X$ with $\|U_n\|\leq 1$, such that for every $x\in \Span(\orb(f))$ the sequence $(\|U_n(x)-U_{n+1}(x)\|_\infty)_{n\in\N}$ is summable. In particular, $(U_n(x))_{n\in\N}$ is Cauchy for every $x\in \Span(\orb(f))$.

Given $y \in X$ and $x\in\Span(\orb(f))$, since $\|U_n\| \leq 1$, we obtain 
\begin{equation*}
    \|U_n y - U_m y\|_{\infty} \leq 2 \|x - y\|_{\infty} + \|U_n x - U_m x\|_{\infty}.
\end{equation*}
This shows that $(U_n y)_{n \in \N}$ is Cauchy for every $y \in X$. It follows that $(U_n)_{n \in \N}$ converges in the strong operator topology to an operator $U: X \rightarrow X$ with $\|U\| \leq 1$.

We claim that $U\neq 0$. Indeed, otherwise \[0=U(f)=\lim_{n\to\infty} U_n(f)=\lim_{n\to\infty} f\ast \Big(\frac{\nu_n^*-(\nu'_n)^*}{2}\Big)= \frac{1}{2} \lim_{n\to\infty} f\ast \Big(\nu_n^*-(\nu'_n)^*\Big)\] which contradicts the assumption that for no subsequence $(k_n)_{n\in\N}$ we have $f\ast (\nu^*_{k_n}-(\nu'_{k_n})^*)\to 0$.

Finally, we consider the operators \[V:=\id_X+U:X\to X,\quad W:=\id_X-U:X\to X.\] Since $U\neq 0$, we have $V\neq W$. Notice that we have $V=\lim_{n\to \infty} T_n+U_n$, resp. $W=\lim_{n\to\infty} T_n-U_n$, in the strong operator topology. Since 
\[T_n+U_n(x)=x\ast \mu^*_n+x\ast \Big(\frac{\nu_n^*-(\nu'_n)^*}{2}\Big)=x\ast\nu_n^*,\quad x\in X, \]
\[T_n-U_n(x)=x\ast \mu^*_n-x\ast \Big(\frac{\nu_n^*-(\nu'_n)^*}{2}\Big)=x\ast (\nu'_n)^*,\quad x\in X\] and $\nu_n,\nu'_n\in \meas(G)$, we get that for all $n\in\N$, \[\|T_n+U_n\|\leq 1,\quad \|T_n-U_n\|\leq 1.\] This implies that $\|V\|\leq 1$ and $\|W\|\leq 1$. But then \[\id_X=\frac{\id_X+U+\id_X-U}{2}=\frac{V+W}{2}\] which contradicts that $\id_X$ is an extreme point of $B_{\mathcal{B}(X)}$, the unit ball of the space of bounded linear operators on $X$ (see e.g. \cite[Theorem 8.28]{OETbook}).
\end{proof}

As a consequence we get that for every countable abelian group $G$, every separable $G$-Banach space admits an equivalent strictly convex $G$-invariant norm. In fact, we prove something stronger.

\begin{corollary}\label{cor:cocompactabelian}
Let $G$ be a topological group that contains a closed abelian subgroup $H$ such that $G/H$ is compact. Then every separable $G$-Banach space admits an equivalent strictly convex $G$-invariant norm.
\end{corollary}
\begin{proof}
Let $X$ be a separable $G$-Banach space. It suffices to show that there is an equivalent $H$-invariant norm since then we can apply Lemma~\ref{lem:cocompactsubgroup}. So suppose that there is no equivalent $H$-invariant strictly convex norm. Then by \Cref{thm:generalequiv}, there is $x\in X$ that is not an extreme point of $\overline{\co(\orb_H(x,-x))}$. Therefore, there are $y\neq z\in\overline{\co(\orb_H(x,-x))}$ such that $x=\frac{y+z}{2}$. Let $(y_n)_{n\in\N}$, resp. $(z_n)_{n\in\N}$ be sequences of elements from $\co(\orb_H(x,-x))$ that converge to $y$, resp. to $z$. Since for each $n$ there are only finitely many elements $h_1,\ldots,h_k\in H$ such that $y_n$ is an absolutely convex combination of elements $h_1\cdot x,\ldots,h_k\cdot x$, and analogously for $z_n$, it follows that there is a countable subgroup $H'\leq H$ such that $y,z\in\overline{\co(\orb_{H'}(x,-x))}$. By \Cref{thm:CD-equiv}, it follows that $H'$ is not strongly Choquet-Deny. However, $H'$ is abelian and so is strongly Choquet-Deny by \Cref{thm:abelianCD}, a contradiction.
\end{proof}

Among finitely generated groups, the simplest example that is neither abelian nor even virtually abelian is the discrete Heisenberg group $\Hei(\Z)$. Here by $\Hei(\Z)$ we mean the subgroup of $\mathrm{GL}_3(\Z)$ consisting of unit upper triangular matrices, i.e. matrices of the form \[\begin{pmatrix}
    1 & a & c \\
    0 & 1 & b \\
    0 & 0 & 1
\end{pmatrix},\] where $a,b,c\in\Z$. Alternatively, it is the group with the presentation 
\begin{equation*} 
\langle a,b,c; [a,c],[b,c],[a,b]c^{-1}\rangle. 
\end{equation*}
Both views allow us to see $\Hei(\Z)$ as $\Z^3$ with the group multiplication defined as \[(a,b,c)\cdot (a',b',c')=(a+a',b+b',c+c'-a'b).\] Given an element $x\in\Hei(\Z)$, we shall identify it with the corresponding element of $\Z^3$ and denote by $x(i)$, $i\in\{1,2,3\}$, its $i$-th coordinate.

\begin{example}\label{ex:Heisenberg}
$\Hei(\Z)$ is not strongly Choquet-Deny.
\end{example}
\begin{proof}
Let us denote by $\FF$ the set of all finite sequences of positive integers without repetitions. For each $v\in\FF$, we denote by $|v|$ its length and by $v(i)$, $i\leq |v|$, its $i$-th element, i.e. each $v\in \mathcal{F}$ is the sequence $(v(i))_{i=1}^{|v|}$, and $v(i)\neq v(j)$, for $i\neq j\leq |v|$.\medskip

\noindent\textbf{Claim 1.} There exists an infinite sequence $(s_n)_{n=1}^\infty\subseteq\Hei(\Z)$ satisfying the following conditions below, where for $v\in\FF$ we denote by $s_v$ the element $\prod_{i\leq |v|} s_{v(i)}$.
\begin{enumerate}
    \item The map $v\in\FF\to s_v\in\Hei(\Z)$ is injective.
    \item For every $v\in\FF$ and $i\leq |v|$, there is no $w\in\FF$ such that $s_v\cdot s_{v(i)}=s_w$.
    \item For every $g\in\Hei(\Z)\setminus \{s_v\colon v\in\FF\}$, there is at most one $n\in\N$ such that $g\cdot s_n\in \{s_v\colon v\in\FF\}$.
\end{enumerate}

For each $n\in\N$, we set \[s_n:=(16^{2^n},2^{2^n},0)\] and we show that this works. We prove (1). Pick $v\neq w\in\FF$ and let us show that $s_v\neq s_w$. First assume that the ranges of $v$ and $w$ are not equal, i.e. $v$ and $w$ are not enumerations of the same finite subset of $\N$. Then \[s_v(1)=\sum_{i\leq |v|} 16^{2^{v(i)}}\neq \sum_{j\leq |w|} 16^{2^{w(j)}}=s_w(1)\] where the inequality follows from the uniqueness of the base-$16$ expansion. Since $s_v(1)\neq s_w(1)$, then in particular $s_v\neq s_w$. So suppose now that $v$ and $w$ enumerate the same finite set and they only differ in the order of that set. Then $s_v(1)=s_w(1)$ and $s_v(2)=s_w(2)$, however, we show $s_v(3)\neq s_w(3)$. From the formula $(a,b,c)\cdot (a',b',c')=(a+a',b+b',c+c'-a'b)$ for the multiplication in $\Hei(\Z)$ we get \[s_v(3)=-\sum_{\{(i,j)\colon 2\leq i\leq |v|,j< i\}} 16^{2^{v(i)}}2^{2^{v(j)}}\] and analogously \[s_w(3)=-\sum_{\{(i,j)\colon 2\leq i\leq |w|,j< i\}} 16^{2^{w(i)}}2^{2^{w(j)}}.\] Let $A:=\{v(i)\colon i\leq |v|\}=\{w(j)\colon j\leq |w|\}\subseteq\N$ be the set that $v$ and $w$ enumerate and let $a\in A$ be the largest element - in the natural order on $\N$ - whose ordering within $v$ and $w$ differs. That is, there are $i_0\leq |v|$ and $j_0\leq |w|$ such that $a=v(i_0)=w(j_0)$ and the sets $A_v:=\{v(i)\colon i<i_0\}\neq \{w(j)\colon j<j_0\}=:A_w$, and moreover $a$ is the largest element with this property. We can write \[s_v(3)-s_w(3)=\sum_{b\in A} 16^{2^b}c(b),\] where for each $b\in A$, $c(b)\in\Z$. Notice that $c(a)\neq 0$. Indeed, we have 

\[c(a)=\sum_{c\in A_w} 2^{2^c} - \sum_{b\in A_v} 2^{2^b} \neq 0,\]

where the last inequality follows from $A_v\neq A_w$ and the uniqueness of the corresponding expansion. Further, by the choice of $a$, for every $a'\in A$ such that $a'>a$, we have $c(a')=0$. On the other hand, for every $a'\in A$ such that $a'<a$ we have 

\[
|c(a')|\leq \sum_{i=1}^{a}2^{2^i}<2\cdot2^{2^a}.
\]

Using that $a  \geq 2$ and that $a - 1 < 2^{2^a - 1}$, it follows that 
\[
\begin{split}
|s_v(3)-s_w(3)|
&=\left|16^{2^a}c(a)+
  \sum_{\{a'\in A\colon a'<a\}}16^{2^{a'}}c(a')\right|\\
&\geq16^{2^a}|c(a)|-
  \sum_{\{a'\in A\colon a'<a\}}16^{2^{a'}}|c(a')|\\
&> 16^{2^a} - (a-1) 16^{2^{a-1}} 2^{2^a+1} \\
&= 2^{4 \cdot 2^a} - (a-1) 2^{3 \cdot 2^a + 1} \\
&= 2^{3 \cdot 2^a + 1} (2^{2^a-1} - (a-1)) > 0.
\end{split}
\]

This shows that $s_v\neq s_w$ and finishes the proof of (1).

We show (2). If for some $v\in\FF$ and $i\leq |v|$ there is $w\in\FF$ such that $s_v\cdot s_{v(i)}=s_w$, then in particular $(s_v\cdot s_{v(i)})(1)=s_w(1)$. $(s_v\cdot s_{v(i)})(1)$ contains the term $2\cdot 16^{2^{v(i)}}$, while $s_w(1)$ contains each term $16^{2^{w(j)}}$, $j\leq |w|$, with multiplicity one. This is a contradiction with the uniqueness of the base-$16$ expansion.

Finally, we show (3). Suppose that for some $g\in\Hei(\Z)\setminus\{s_v\colon v\in\FF\}$ there are $n\neq m\in\N$ and $v,w\in\FF$ such that $g\cdot s_n=s_v$ and $g\cdot s_m=s_w$. Thus, $g=s_v\cdot s_n^{-1}=s_w\cdot s_m^{-1}$. Since $s_n^{-1} = (-16^{2^n}, - 2^{2^n}, -16^{2^n} 2^{2^n})$ and $s_m^{-1} = (-16^{2^m}, - 2^{2^m}, -16^{2^m} 2^{2^m})$, we get 
\[\left(\sum_{i\leq |v|} 16^{2^{v(i)}}\right)-16^{2^n}=(s_v\cdot s_n^{-1})(1)=g(1)= (s_w\cdot s_m^{-1})(1)=\left(\sum_{j\leq |w|} 16^{2^{w(j)}}\right)-16^{2^m}.\] By the uniqueness of the base-$16$ expansion and since $n\neq m$ we get that there are $i_0\leq |v|$, resp. $j_0\leq |w|$ such that $v(i_0)=n$, resp. $w(j_0)=m$. Consequently, we get that $A:=\{v(i)\colon i\leq|v|\}\setminus\{n\}=\{w(j)\colon j\leq |w|\}\setminus\{m\}$, in particular $A\cap\{n,m\}=\emptyset$, and we have 
\begin{equation}\label{eq:g(1)}
g(1)=\sum_{a\in A} 16^{2^a}  
\end{equation}
and an analogous argument gives that 
 \begin{equation}\label{eq:g(2)}
g(2)=\sum_{a\in A} 2^{2^a}.  
 \end{equation}
 
 In order to reach a contradiction, we show that there exists an enumeration $t\in\FF$ of $A$ such that $g=s_t$. We have 
 
\begin{equation}\label{eq:g(3)}
\begin{split}
    g(3) & =-\left(\sum_{i=2}^{|v|} \sum_{j<i} 16^{2^{v(i)}}2^{2^{v(j)}}\right)+\sum_{1\leq k\leq |v|; k\neq i_0} 16^{2^n} 2^{2^{v(k)}}\\ &=-\left(\sum_{i=2}^{|w|} \sum_{j<i} 16^{2^{w(i)}}2^{2^{w(j)}}\right)+\sum_{1\leq k\leq |w|; k\neq j_0} 16^{2^m} 2^{2^{w(k)}}. 
\end{split}  
\end{equation}
Let $t\in\FF$ be $v$ with $n=v(i_0)$ omitted, that is, $|t|=|v|-1$, $t(i)=v(i)$, for $i<i_0$, and $t(i)=v(i+1)$ for $i \geq i_0$. Analogously, let $u\in\FF$ be $w$ with $m=w(j_0)$ omitted. Set also $A_v:=\{v(i)\colon i_0<i\leq |v|\}$ and $A_w:=\{w(j)\colon j_0<j\leq |w|\}$.

From \eqref{eq:g(3)} we get \[\begin{split} g(3) &=s_t(3)+\sum_{a\in A_v} (16^{2^n}2^{2^a}-16^{2^a}2^{2^n})\\ &=s_u(3)+\sum_{b\in A_w} (16^{2^m}2^{2^b}-16^{2^b}2^{2^m})\end{split}\]  which implies
\begin{equation}\label{eq:Heisenberg(3)}
 s_t(3)-s_u(3)= \sum_{b\in A_w} (16^{2^m}2^{2^b}-16^{2^b}2^{2^m}) - \sum_{a\in A_v} (16^{2^n}2^{2^a}-16^{2^a}2^{2^n}).
\end{equation}
Let \[D=\{(a,b)\in A^2\colon \exists l<l',k>k'\; \big(a=t(l)=u(k)\wedge b=t(l')=u(k')\big)\}\] be the set of pairs of distinct elements from $A$ on which the orderings $t$ and $u$ differ. It follows that
\begin{equation}\label{eq:almostfinalHeisenberg}
s_t(3)-s_u(3)=\sum_{(c,d)\in D} (16^{2^c}2^{2^d}-16^{2^d} 2^{2^c}),
\end{equation}

so 
\begin{equation}\label{eq:finalHeisenberg}
\sum_{(c,d)\in D} (16^{2^c}2^{2^d}-16^{2^d} 2^{2^c})-\sum_{b\in A_w} (16^{2^m}2^{2^b}-16^{2^b}2^{2^m}) + \sum_{a\in A_v} (16^{2^n}2^{2^a}-16^{2^a}2^{2^n})=0.
\end{equation}

Given $k,l\in\N$, denote by $D(k,l)$ the term $16^{2^k} 2^{2^l}-16^{2^l} 2^{2^k}$, and notice that $D(k,l)=-D(l,k)$ and $D(k,k)=0$.\medskip

\noindent\textit{Subclaim 1a.} For any finite set $P\subseteq \{(k,l)\in\N^2\colon k>l\}$, if $\sum_{(k,l)\in P} \varepsilon_{k,l}\cdot D(k,l)=0$, where $(\varepsilon_{k,l})_{(k,l)\in P}\subseteq \{-1,0,1\}$, then $\varepsilon_{k,l}=0$ for all $(k,l)\in P$.

Suppose for a moment that Subclaim 1a has been proved. Notice that \eqref{eq:finalHeisenberg} can be written in the form as in the statement of Subclaim 1a since each term $D(k,l)$ from \eqref{eq:finalHeisenberg}, where $k<l$ can be replaced with $-D(l,k)$. By Subclaim 1a we therefore obtain that each term $D(k,l)$ is with the zero coefficient, thus both sides of \eqref{eq:almostfinalHeisenberg} are zero. In particular, $s_t(3)-s_u(3)=0$, so $g(3)=s_t(3)=s_u(3)$. Since by \eqref{eq:g(1)} and \eqref{eq:g(2)} we also have $g(1)=s_t(1)=s_u(1)$ and $g(2)=s_t(2)=s_u(2)$, we conclude that $g=s_t=s_u$, contradicting that $g\notin\{s_r\colon r\in\FF\}$. Thus, in order to finish the proof of (3) and therefore also of Claim 1, we need to prove Subclaim 1a.

We prove it now. Denote by $\mathbb{S}$ the sum $\sum_{(k,l)\in P} \varepsilon_{k,l}\cdot D(k,l)$, which by assumption is zero. Suppose on the contrary that not all the coefficients $\varepsilon_{k,l}$ are zero. Let $k_M$ be the largest integer such that for some $l'<k_M$, $(k_M,l')\in P$ and $\varepsilon_{k_M,l'}\neq 0$. Write 
\begin{equation}\label{eq:subclaim1-1}
 \sum_{\{l<k_M\colon (k_M,l)\in P\}} \eps_{k_M,l} D(k_M,l)=16^{2^{k_M}} B-2^{2^{k_M}}C ,  
\end{equation}
where \[B=\sum_{\{l<k_M\colon (k_M,l)\in P\}} \varepsilon_{k_M,l}\cdot 2^{2^l},\quad C=\sum_{\{l<k_M\colon (k_M,l)\in P\}} \varepsilon_{k_M,l}\cdot 16^{2^l}.\] By the uniqueness of the binary expansion, $B\neq 0$.

We also denote by $R$ the `remainder'  \[\mathbb{S} - 16^{2^{k_M}}B + 2^{2^{k_M}}C=\sum_{\{(k,l) \in P \colon k<k_M\}} \varepsilon_{k,l} \cdot D(k,l)=:R.\] Clearly, 

\begin{equation}\label{eq:subclaim1-2}
\begin{split}
|C|
\leq\sum_{\{l<k_M\colon(k_M,l)\in P\}}16^{2^l}
\leq\sum_{l<k_M}16^{2^l}
<2\cdot16^{2^{k_M-1}}=2\cdot2^{2^{k_M+1}}.
\end{split}
\end{equation}

Next, notice that for every $k>l$ we have $0<D(k,l)<16^{2^k}\cdot 2^{2^l}$, thus, 

\begin{equation}\label{eq:subclaim1-3}
|R|\leq \sum_{\{(k,l) \in P\colon k<k_M\}} 16^{2^k}2^{2^l}\leq \left(\sum_{k<k_M} 16^{2^k}\right)\cdot\left(\sum_{l<k_M} 2^{2^l}\right)<2\cdot 2^{2^{k_M+1}}\cdot \left(\sum_{l<k_M} 2^{2^l}\right).
\end{equation}

We show that $|16^{2^{k_M}}B|> |-2^{2^{k_M}}C +R|$. Since $\mathbb{S}=16^{2^{k_M}} B-2^{2^{k_M}}C+R$, this will show that $\mathbb{S}\neq 0$ contrary to the assumption. By \eqref{eq:subclaim1-2} and \eqref{eq:subclaim1-3}, we have \[\begin{split}|-2^{2^{k_M}}C +R| &< 2^{2^{k_M}}\cdot 2\cdot 2^{2^{k_M+1}} + 2\cdot 2^{2^{k_M+1}}\cdot  \left(\sum_{l<k_M} 2^{2^l}\right)\\ &= 2\cdot 2^{2^{k_M+1}}\cdot \left(\sum_{l\leq k_M} 2^{2^l}\right)\leq 4\cdot 2^{2^{k_M+1} + 2^{k_M}}\\ &=2^{3\cdot 2^{k_M}+2\leq 2^{4\cdot 2^{k_M}}=} 16^{2^{k_M}}\leq |16^{2^{k_M}} B|.\end{split}\] This finishes the proof of Claim 1.\medskip

Now let $(N_n)_{n=1}^\infty$ be a sequence of disjoint subsets of $\N$ such that $\N=\bigcup_{n=1}^\infty N_n$, and for each $n\in\N$, $|N_n|=2^{2^n}$. We set $S_n:=\{s_i\colon i\in N_n\}$, for each $n\in\N$, and we let $\mu_n\in\prob(\Hei(\Z))$ be the finitely supported probability measure uniformly distributed on $S_n$, i.e. \[\mu_n(g):=\begin{cases}
    \frac{1}{|S_n|} & g\in S_n,\\
    0 & \text{otherwise}.
\end{cases}\]
Further, for each $n\in\N$ we choose subsets $S_n^+,S^-_n\subseteq S_n$ such that $S_n=S_n^+\cup S_n^-$ and $|S_n^+|=|S_n^-|=2^{2^n-1}$; in particular, $S_n^+\cap S_n^-=\emptyset$. For each $n\in\N$, we let $\nu_n$, resp. $\nu'_n$, be the finitely supported probability measure uniformly distributed on $S_n^+$, resp. on $S_n^-$. Notice that then we have $\mu_n=\frac{\nu_n+\nu'_n}{2}$, for every $n\in\N$.\medskip

\noindent\textbf{Claim 2.} There exists $f\in\ell_\infty(\Hei(\Z))$ satisfying (1) and (2) of Definition~\ref{def:stronglyCD} with respect to $(\mu_n)_n$, $(\nu_n)_n$, and $(\nu'_n)_n$ such that \[f\ast \big(\nu_n^*-(\nu'_n)^*\big)\to 2 \delta_{1_{\Hei(\Z)}},\] where $\delta_{1_{\Hei(\Z)}}$ is the indicator function of the singleton $\{1_{\Hei(\Z)}\}$.

We set $\Ss:=\{s_v\colon v\in\FF\}\subseteq\Hei(\Z)$. Before defining $f\in\ell_\infty(\Hei(\Z))$ as required by the claim, we introduce several auxiliary functions. First, we define the \emph{sign function} $\sigma: \{s_n\colon n\in\N\}\to\{-1,1\}$ as follows. For any $i\in\N$ we set

\[
\sigma(s_i):=\begin{cases}
    +1 & \exists n\in\N\; \nu_n(s_i)>0,\\
    -1 & \exists n\in\N\; \nu'_n(s_i)>0.
\end{cases}
\]

Next, we define a function $f': \Ss\to \R$ by the following formula. We set \[f'(s_v):=\sum_{i=1}^{|v|} \sum_{n=1}^\infty \frac{2^{2^n}}{2^n} \mu_n(s_{v(i)}),\quad v\in\FF.\] Note that only finitely many terms in the above formula are non-zero, so the function is well-defined. Notice that for each $n\in\N$ and $s_i\in S_n$, the definition gives $f'(s_i)=\frac{1}{2^n}$.

Finally, we define $f:\Hei(\Z)\to [-1,1]$ as follows. We set $f(1_{\Hei(\Z)})=0$ and given $g\in\Hei(\Z)\setminus\{1_{\Hei(\Z)}\}$, we define \[f(g):=\begin{cases}
    \sigma(s_{v(1)})\cdot \max\{0,1- f'(s_v)\} & \exists v\in\FF\; g=s_v,\\
    0 & g\notin \Ss.
\end{cases}\]
Notice that $\|f\|_{\infty} = 1$ as $|f(s_i)| = 1 - \frac{1}{2^n}$ for every $s_i \in S_n$.

\noindent\textit{Subclaim 2a.} $f$ is $\frac{1}{2^{n-1}}$-approximately $\mu_n$-harmonic for each $n$, i.e. $f$ satisfies condition (1) of Definition~\ref{def:stronglyCD}.

Fix $n$ and $g\in \Hei(\Z)$. We estimate $|f(g)-\sum_{h\in S_n} \mu_n(h)f(gh)|$. If $g\notin\Ss$, then $f(g)=0$ and by Claim 1 (3) there is at most one element $h\in S_n$ such that $gh\in\Ss$. Thus, \[\left|f(g)-\sum_{h\in S_n} \mu_n(h)f(gh)\right|\leq \frac{1}{2^{2^n}}< \frac{1}{2^n}.\] If $g\in\Ss \setminus \{ 1_{\mathbb{H}_3(\Z)}\}$ and $f(g)=0$, then $f(gh)=0$, for every $h\in S_n$, thus $|f(g)-\sum_{h\in S_n} \mu_n(h)f(gh)|=0$. So the remaining case is if $f(g)\neq 0$, and thus in particular $g\in\Ss$. Then for every $h\in S_n$ either $gh\in\Ss$, or not. In the former case, $|f(g)-f(gh)|\leq \frac{1}{2^n}.$ In the latter case, $f(gh)=0$, which happens only if $g=s_v$ and $h=s_{v(i)}$, for some $v\in\FF$ and $i\leq |v|$. Since $f(g)=f(s_v)\neq 0$, the number of such $h$'s is bounded by $2^n$. Thus we have \[\begin{split} \left|f(g)-\sum_{h\in S_n} \mu_n(h)f(gh)\right| &\leq \sum_{h\in S_n}\mu_n(h)|f(g)-f(gh)|\\ &\leq \frac{2^n}{2^{2^n}}+ \frac{1}{2^n}\leq \frac{1}{2^{n-1}}.\end{split}\]

\noindent\textit{Subclaim 2b.} $f$ satisfies condition (2) of Definition~\ref{def:stronglyCD} and $f\ast \big(\nu_n^*-(\nu'_n)^*\big)\to 2\delta_{1_{\Hei(\Z)}}$.

Fix $n$ and first take any $g\in\Hei(\Z)\setminus\{1_{\Hei(\Z)}\}$. We estimate $|f\ast \big(\nu_n^*-(\nu'_n)^*\big)(g)|$. If $g\notin\Ss$, then again  by Claim 1 (3) there is at most one element $h\in S_n$ such that $gh\in\Ss$. If there is no such $h\in S_n$, then $f\ast \big(\nu_n^*-(\nu'_n)^*\big)(g)=0$, so suppose there is exactly one such $h$ and assume that $\nu_n(h)\neq 0$, the case `$\nu'_n(h)\neq 0$' being similar. Then, \[\left|f\ast \big(\nu_n^*-(\nu'_n)^*\big)(g)\right|=\left|\sum_{k\in S_n^+} \nu_n(k)f(gk)\right|=\left|\nu_n(h)f(gh)\right|\leq \frac{1}{2^{2^n-1}}\leq \frac{1}{2^n}.\]
Next we assume that $g\in\Ss$. There is $v\in\FF$ such that $g=s_v$. Define $S_n^v:=\{s_i\in S_n\colon \exists j\leq |v|\; v(j)=i\}$ and $S_n^{-v}:=S_n\setminus S_n^v$. If $|S_n^v|\geq 2^n$, then $f'(g)=f'(s_v)\geq 1$, so $f(g)=f(s_v)=0$. Then clearly for every $h\in S_n$, $f(gh)=0$, so $f\ast \big(\nu_n^*-(\nu'_n)^*\big)(g)=0$.

So we assume that $|S_n^v|<2^n$. Then we have for every $h\in S_n^v$, by Claim 1 (2), that $f(gh)=0$, and for every $h\in S_n^{-v}$ that $|f(gh)|\leq 1-\frac{1}{2^n}$. Moreover, since $\sum_{h\in S_n} \sigma(h)=0$, we get $\sum_{h\in S_n^v}\sigma(h)=-\sum_{h'\in S_n^{-v}} \sigma(h')$. In particular, since $|S_n^v|<2^n$, we have $|\sum_{h\in S_n^{-v}} \sigma(h)|<2^n$. Therefore, we get 
\[\left|f\ast \big(\nu_n^*-(\nu'_n)^*\big)(g)\right|  =\left| \sum_{h\in S_n^{-v}} \frac{\sigma(h)}{2^{2^n-1}} f(gh)\right| \leq 2^n \left(1-\frac{1}{2^n}\right)\frac{1}{2^{2^n-1}}\leq \frac{2^n}{2^{2^n-1}} \leq \frac{1}{2^{n-1}}.\]

It remains to estimate $|f\ast \big(\nu_n^*-(\nu'_n)^*\big)(1_{\Hei(\Z)})|$. Notice that $|f(h)|=|f(h')|=1-\frac{1}{2^n}$, for each $h,h'\in S_n$, and moreover, for $h\in S_n$, $f(h)>0$ if and only if $h\in S_n^+$. Therefore, we have 
\begin{equation*}
\begin{aligned}
\left|f\ast\big(\nu_n^*-(\nu'_n)^*\big)(1_{\Hei(\Z)})\right|
&=\sum_{h\in S_n^+}\nu_n(h)\left(1-\frac{1}{2^n}\right)
+\sum_{h'\in S_n^-}\nu'_n(h')\left(1-\frac{1}{2^n}\right)\\
&=2\left(1-\frac{1}{2^n}\right).
\end{aligned}
\end{equation*}
Therefore, we got that \[\begin{split} &\lim_{n\to\infty} \left|f\ast \big(\nu_n^*-(\nu'_n)^*\big)(1_{\Hei(\Z)})\right|=2\\ &\lim_{n\to\infty} \left|f\ast \big(\nu_n^*-(\nu'_n)^*\big)(g)\right|=0,\quad g\in\Hei(\Z)\setminus\{1_{\Hei(\Z)}\}.\end{split}\]
This shows that $f\ast \big(\nu_n^*-(\nu'_n)^*\big)\to 2\delta_{1_{\Hei(\Z)}}$. Condition (2) of Definition~\ref{def:stronglyCD} is immediate.\medskip

We conclude that $f$ along with the measures $(\mu_n)_n$ and decompositions $(\mu_n=\frac{\nu_n+\nu'_n}{2})_n$ witnesses that $\Hei(\Z)$ is not positively strongly Choquet-Deny, thus it is not strongly Choquet-Deny.
\end{proof}

The following is a complete classification of finitely generated strongly Choquet-Deny groups.

\begin{theorem}
A finitely generated group is strongly Choquet-Deny if and only if it is virtually abelian.
\end{theorem}
\begin{proof}
By either Corollary~\ref{cor:cocompactabelian} or \Cref{thm:abelianCD} combined with Proposition~\ref{prop:closureofCD}.(2), we get that virtually abelian groups are strongly Choquet-Deny (even when they are not finitely generated). So we focus on the converse.

First, by \cite[Theorem 1]{FHTV}, if a finitely generated group $G$ is not virtually nilpotent, then it is not Choquet-Deny. By Proposition~\ref{prop:CDversions}, such $G$ is not strongly Choquet-Deny.

Therefore, we only need to treat finitely generated virtually nilpotent groups that are not virtually abelian. Fix such a group $G$. We claim that $G$ contains $\Hei(\Z)$ as a subgroup. Indeed, first $G$ contains a nilpotent subgroup of finite index, which is therefore still finitely generated and not virtually abelian. By the proof of \cite[Theorem 12]{HoReRoTh05}, such a subgroup contains $\Hei(\Z)$ as a subgroup. By Example~\ref{ex:Heisenberg}, $\Hei(\Z)$ is not strongly Choquet-Deny, and by Proposition~\ref{prop:closureofCD}, the strongly Choquet-Deny property passes to subgroups, so its negation passes to overgroups. Therefore, $G$ is not strongly Choquet-Deny. This finishes the proof.
\end{proof}

\appendix
\section{Invariant strictly convex renormings of $c$}

In this appendix, we show that the lamplighter group admits an action on $c$ by linear isometries which does not admit an equivalent invariant strictly convex norm. As a consequence, there is no strictly convex renorming on $c$ which is invariant under the canonical action of $\Iso(c)$. This corrects \cite[Theorem 4.8]{Do26} and answers a question from \cite{AFGR19}.

\begin{theorem} \label{theorem:lamplighter-c}
Let $\theta:\Z\rightarrow\N$ be a bijection. Define $R,U\in\Iso(c)$ by
\begin{equation*}
    (Rx)_{\theta(k)}:=x_{\theta(k-1)}
\end{equation*}
for every $k\in\Z$ and
\begin{equation*}
    (Ux)_{\theta(k)}:=
    \begin{cases}
        -x_{\theta(0)}, & \mbox{if } k=0,\\
        x_{\theta(k)}, & \mbox{if } k\not=0,
    \end{cases}
\end{equation*}
for every $k\in\Z$. Let
\begin{equation*}
    G:=\langle R,U\rangle\leq\Iso(c).
\end{equation*}
Then, $G$ is isomorphic to the lamplighter group and the canonical action of $G$ on $c$ admits no equivalent $G$-invariant strictly convex norm.
\end{theorem}

\begin{proof} The operator $R$ permutes the coordinates of every element of $c$, while $U$ changes the sign of only the coordinate indexed by $\theta(0)$. Therefore, $R,U\in\Iso(c)$. For every $k\in\Z$, define $U_k:=R^kUR^{-k}$. The operator $U_k$ changes the sign of precisely the coordinate indexed by $\theta(k)$. Consequently, $U_k^2=\id$ for every $k\in\Z$ and $U_jU_k=U_kU_j$ for every $j,k\in\Z$. Moreover, a finite product of distinct operators $U_k$ is the identity if and only if the product is empty. Thus, $N:=\langle U_k:k\in\Z\rangle$ is isomorphic to the direct sum $\bigoplus_{k\in\Z}\Z_2$. Furthermore, $RU_kR^{-1}=U_{k+1}$ for every $k\in\Z$. Hence, $N$ is normal in $G$ and the quotient $G/N$ is generated by the coset $RN$. Since no non-zero power of $R$ belongs to $N$, the quotient $G/N$ is isomorphic to $\Z$. Therefore,
\begin{equation*}
    G\cong\left(\bigoplus_{k\in\Z}\Z_2\right)\rtimes\Z,
\end{equation*}
which is the lamplighter group.

We now prove that the canonical action of $G$ on $c$ admits no equivalent $G$-invariant strictly convex norm. Let $x:=\mathbf{1}=(1,1,1,\ldots)$  and set $y:=Ux=\mathbf{1}-2e_{\theta(0)}$, where $(e_n)_{n\in\N}$ denotes the canonical unit vectors. In particular, $x\not=y$. Set $z:=\frac{x+y}{2}=\mathbf{1}-e_{\theta(0)}$. 
Suppose that $\nn{\cdot}$ is an equivalent $G$-invariant norm on $c$. Since $y=Ux$ and $U\in G$, we have $\nn{y}=\nn{x}$. Moreover, for every $k\in\Z$, $R^kz=\mathbf{1}-e_{\theta(k)}$, and, by $G$-invariance, $\nn{R^kz}=\nn{z}$. For every $m\in\N$, define
\begin{equation*}
    z_m:=\frac{1}{m}\sum_{k=0}^{m-1}R^kz
    =\mathbf{1}-\frac{1}{m}\sum_{k=0}^{m-1}e_{\theta(k)}.
\end{equation*}
We have $\|z_m-x\|_{\infty}=\frac{1}{m}$ for every $m\in\N$. Therefore, $z_m\rightarrow x$ in the canonical norm of $c$. Since $\nn{\cdot}$ is equivalent to the canonical norm, $\nn{z_m}\longrightarrow\nn{x}$. On the other hand, by convexity and $G$-invariance,
\begin{equation*}
    \nn{z_m}\leq\frac{1}{m}\sum_{k=0}^{m-1}\nn{R^kz}=\nn{z}
\end{equation*}
for every $m\in\N$. Passing to the limit, we obtain $\nn{x}\leq\nn{z}$. Since $z=\frac{x+y}{2}$ and $\nn{x}=\nn{y}$, convexity also gives
\begin{equation*}
    \nn{z}=\nn{\frac{x+y}{2}}\leq\frac{\nn{x}+\nn{y}}{2}=\nn{x}.
\end{equation*}
Therefore, $\nn{x}=\nn{y}=\nn{\frac{x+y}{2}}$. Since $x\not=y$, the norm $\nn{\cdot}$ cannot be strictly convex.
\end{proof}

In \cite[Theorem 4.8]{Do26} it is claimed that the Banach space $c$ satisfies the GSCR (a $G$-Banach space satisfies the group-invariant strictly convex renormable property (GSCR, for short) if for every group $G$ and every continuous action $G \curvearrowright X$ by linear isometries, there exists a $G$-invariant strictly convex renorming of $X$). However, there is a mistake in its proof. The map $J:c\longrightarrow(c_0,\nn{\cdot})\oplus_{\ell_2}\R$ defined by
\begin{equation*}
    J((x_n)_{n\in\N})
    :=
    \left((x_n-\lim_{n\rightarrow\infty}x_n)_{n\in\N}\right)
    \oplus
    \lim_{n\rightarrow\infty}x_n
\end{equation*}
is not $\Iso(c)$-equivariant as claimed there. Indeed, consider the constant sequence $x:=\mathbf{1}=(1,1,1,\ldots)$ and let $U\in\Iso(c)$ be defined by 
\begin{equation*} 
U(x_1,x_2,x_3,\ldots):=(-x_1,x_2,x_3,\ldots). 
\end{equation*}
Then, $J(x)=0\oplus1$ and, under the action described in the proof of \cite[Theorem 4.8]{Do26}, $U\cdot J(x)=0\oplus1$. On the other hand,
 $Ux=(-1,1,1,\ldots)$ and therefore
\begin{equation*}
    J(Ux)=(-2,0,0,\ldots)\oplus1\not=U\cdot J(x).
\end{equation*}
Theorem \ref{theorem:lamplighter-c} provides the following result, which in turn gives a negative answer to \cite[Question 8.1]{AFGR19}.

\begin{corollary} \label{c-no-equivalent-SC-renorming}
There is no strictly convex renorming on $c$ which is invariant under the canonical action of $\Iso(c)$.
\end{corollary}

\section*{Declaration on AI use}
The project was initiated when the first named author discovered a gap in \cite[Theorem 4.8]{Do26}, noticed that $c$ does not admit an invariant strictly convex renorming with respect to the full isometry group (this is the content of the appendix). Then, by guided prompting using ChatGPT Plus 5.6, found that the free 2-step nilpotent group on countably many generators and the Heisenberg group admit actions on separable Banach spaces without invariant strictly convex renormings. The current paper builds on that AI output, which we make available at \cite{AI} and pushes it much further. All the other results were solely developed by the authors, including isolating the main notion of the paper, relating it to Choquet-Deny groups, and obtaining the classification of finitely generated groups with this property. The paper is fully human-written. Example~\ref{ex:Heisenberg} is the part of the paper most influenced by \cite{AI}. Some of the tedious computations from Claim 1 of the proof of Example~\ref{ex:Heisenberg} were AI-assisted.
\bibliographystyle{abbrv}

\bibliography{ref}

\begin{thebibliography}{10}

\bibitem{AFGR19}
L.~Antunes, V.~Ferenczi, S.~Grivaux, and C.~Rosendal.
\newblock Light groups of isomorphisms of {B}anach spaces and invariant {LUR} renormings.
\newblock {\em Pacific J. Math.}, 301(1):31--54, 2019.

\bibitem{AI}
{ChatGPT} output.
\newblock \href{https://drive.google.com/file/d/10_IHTfkyz68qcOtbAar0pp5j1DJypfwV/view?usp=sharing}{Link}, 2026.

\bibitem{ChoDe60}
G.~Choquet and J.~Deny.
\newblock Sur l'{\'e}quation de convolution {{\(\mu = \mu * \sigma\)}}.
\newblock {\em C. R. Acad. Sci., Paris}, 250:799--801, 1960.

\bibitem{Do26}
M.~Doucha.
\newblock Invariant strictly convex renormings.
\newblock {\em J. Funct. Anal.}, 290(11):27, 2026.
\newblock Id/No 111428.

\bibitem{DruKabook}
C.~Dru\c~tu and M.~Kapovich.
\newblock {\em Geometric group theory}, volume~63 of {\em American Mathematical Society Colloquium Publications}.
\newblock American Mathematical Society, Providence, RI, 2018.
\newblock With an appendix by Bogdan Nica.

\bibitem{OETbook}
T.~Eisner, B.~Farkas, M.~Haase, and R.~Nagel.
\newblock {\em Operator theoretic aspects of ergodic theory}, volume 272 of {\em Grad. Texts Math.}
\newblock Cham: Springer, 2015.

\bibitem{Er05}
A.~Erschler.
\newblock Boundary behavior for groups of subexponential growth.
\newblock {\em Ann. Math. (2)}, 160(3):1183--1210, 2005.

\bibitem{FR11}
V.~Ferenczi and C.~Rosendal.
\newblock Displaying {P}olish groups on separable {B}anach spaces.
\newblock {\em Extracta Math.}, 26(2):195--233, 2011.

\bibitem{FR13}
V.~Ferenczi and C.~Rosendal.
\newblock On isometry groups and maximal symmetry.
\newblock {\em Duke Math. J.}, 162(10):1771--1831, 2013.

\bibitem{FHTV}
J.~Frisch, Y.~Hartman, O.~Tamuz, and P.~Vahidi~Ferdowsi.
\newblock Choquet--deny groups and the infinite conjugacy class property.
\newblock {\em Ann. Math. (2)}, 190(1):307--320, 2019.

\bibitem{Fur63}
H.~Furstenberg.
\newblock A {Poisson} formula for semi-simple {Lie} groups.
\newblock {\em Ann. Math. (2)}, 77:335--386, 1963.

\bibitem{renorming-book}
A.~J. Guirao, V.~Montesinos, and V.~Zizler.
\newblock {\em Renormings in {B}anach spaces---a toolbox}, volume~75 of {\em Monografie Matematyczne [Mathematical Monographs]}.
\newblock Birkh\"auser/Springer, Cham, [2022] \copyright 2022.

\bibitem{HoReRoTh05}
D.~F. Holt, S.~Rees, C.~E. R{\"o}ver, and R.~M. Thomas.
\newblock Groups with context-free co-word problem.
\newblock {\em J. Lond. Math. Soc., II. Ser.}, 71(3):643--657, 2005.

\bibitem{Izu02}
M.~Izumi.
\newblock Non-commutative {Poisson} boundaries and compact quantum group actions.
\newblock {\em Adv. Math.}, 169(1):1--57, 2002.

\bibitem{Ja04}
W.~Jaworski.
\newblock Countable amenable identity excluding groups.
\newblock {\em Can. Math. Bull.}, 47(2):215--228, 2004.

\bibitem{JaNe07}
W.~Jaworski and M.~Neufang.
\newblock The {C}hoquet-{D}eny equation in a {B}anach space.
\newblock {\em Canad. J. Math.}, 59(4):795--827, 2007.

\bibitem{Ka00}
V.~A. Kaimanovich.
\newblock The {Poisson} formula for groups with hyperbolic properties.
\newblock {\em Ann. Math. (2)}, 152(3):659--692, 2000.

\bibitem{KaVe83}
V.~A. Kaimanovich and A.~M. Vershik.
\newblock Random walks on discrete groups: {Boundary} and entropy.
\newblock {\em Ann. Probab.}, 11:457--490, 1983.

\bibitem{Lal-book}
S.~P. Lalley.
\newblock {\em Random walks on infinite groups}, volume 297 of {\em Grad. Texts Math.}
\newblock Cham: Springer, 2023.

\bibitem{Wi90}
G.~A. Willis.
\newblock Probability measures on groups and some related ideals in group algebras.
\newblock {\em J. Funct. Anal.}, 92(1):202--263, 1990.

\bibitem{Woe-book}
W.~Woess.
\newblock {\em Random walks on infinite graphs and groups}, volume 138 of {\em Camb. Tracts Math.}
\newblock Cambridge: Cambridge University Press, 2000.

\bibitem{Yad-book}
A.~Yadin.
\newblock {\em Harmonic functions and random walks on groups}, volume 213 of {\em Camb. Stud. Adv. Math.}
\newblock Cambridge: Cambridge University Press, 2024.

\bibitem{Zim78}
R.~J. Zimmer.
\newblock Amenable ergodic group actions and an application to {Poisson} boundaries of random walks.
\newblock {\em J. Funct. Anal.}, 27:350--372, 1978.

\end{thebibliography}
\end{document}